\documentclass
{amsart}
\usepackage[utf8]{inputenc}
\usepackage{amssymb}
\usepackage{amsmath}
\usepackage{mathtools}
\usepackage{mathrsfs} 
\usepackage{calligra} 
\usepackage{ifpdf}
\usepackage{vmargin}

\ifpdf
\usepackage[pdftex]{hyperref}
\else
\usepackage[hypertex]{hyperref}
\fi

\hypersetup{
citecolor=green,
menucolor=red,
citebordercolor=0 1 0,
linkbordercolor=1 0 0,
pdfpagemode=UseOutlines,
pdftitle={High speed excited multi-solitons in nonlinear Schrödinger
equations},
pdfauthor={Raphaël Côte <cote@math.polytechnique.fr>, Stefan LE COZ <lecoz@ann.jussieu.fr>},
pdfsubject={Multi-solitons for NLS},
pdfkeywords={multi-soliton, excited states, nonlinear Schrödinger equation} 
}

\newtheorem{thm}{Theorem}
\newtheorem{prop}{Proposition}
\newtheorem{lem}[prop]{Lemma}
\newtheorem{rmk}[prop]{Remark}
\newtheorem{claim}[prop]{Claim}
\newtheorem{cor}[prop]{Corollary}
\newtheorem{defn}[prop]{Definition}

\DeclarePairedDelimiter{\nld}{\lVert}{\rVert_{\ld}}
\DeclarePairedDelimiter{\nhu}{\lVert}{\rVert_{\hu}}
\DeclarePairedDelimiter{\nldd}{\lVert}{\rVert^{2}_{\ld}}
\DeclarePairedDelimiter{\nhud}{\lVert}{\rVert^{2}_{\hu}}

\newcommand{\psld}[2]{\left( #1,#2 \right)_{L^2(\Rd)}}

\newcommand{\dual}[2]{\left< #1,#2 \right>}

\newcommand{\e}{\varepsilon}
\newcommand{\g}{\gamma}
\newcommand{\s}{\sigma}
\newcommand{\w}{\omega}

\newcommand{\D}{\Delta}

\newcommand{\C}{\mathbb{C}}
\newcommand{\N}{\mathbb{N}}
\newcommand{\R}{\mathbb{R}}
\newcommand{\Rd}{\mathbb{R}^d}

\newcommand{\ld}{L^2(\Rd)}
\newcommand{\hu}{H^1(\Rd)}

\DeclareMathAlphabet{\mathpzc}{OT1}{pzc}{m}{it}
\DeclareMathAlphabet{\mathcalligra}{T1}{calligra}{m}{n}
\renewcommand{\Re}{\mathscr{R}\!\mathcalligra{e}\,}
\renewcommand{\Im}{\mathscr{I}\hspace*{-5pt}\mathcalligra{m}\,}

\renewcommand{\leq}{\leqslant}
\renewcommand{\geq}{\geqslant}
\renewcommand{\le}{\leqslant}
\renewcommand{\ge}{\geqslant}

\newcommand{\intrd}{\int_{\Rd}}

\newcommand{\m}[1]{\mathbb{#1}}
\newcommand{\q}[1]{\mathscr{#1}}
\DeclareMathOperator{\Id}{\mathrm{Id}}
\DeclareMathOperator{\Sp}{\mathrm{Sp}}

\begin{document}

\title[High speed excited multi-solitons in NLS]{High speed excited multi-solitons in nonlinear Schrödinger
equations}

\subjclass[2010]{35Q55,(35Q51,37K40)}
\keywords{Multi-solitons, Nonlinear Schrödinger equations, Excited states}

\date\today

\author{Raphaël Côte}
\address[Raphaël Côte]{CMLS, École Polytechnique, 91128 Palaiseau Cedex, France}
\email{cote@math.polytechnique.fr}

\author{Stefan Le Coz}
\address[Stefan Le Coz]{Laboratoire Jacques-Louis Lions, Université Pierre et Marie Curie, Boîte courrier 187, 75252 Paris Cedex 05, France}
\email{lecoz@ann.jussieu.fr}
\thanks{\emph{Correspondind Author:} S. Le Coz\\ S. Le Coz is supported by ANR project ESONSE}

\begin{abstract}
We consider the nonlinear Schrödinger equation in $\m R^d$
\[ 
i \partial_t u + \Delta u + f(u) =0. 
\]
For $d \ge 2$, this equation admits travelling wave solutions of the form $e^{i\omega t} \Phi(x)$ (up to a Galilean transformation), where $\Phi$ is a fixed profile, solution to $- \Delta \Phi + \w \Phi = f(\Phi)$, but \emph{not the ground state}. This kind of profiles are called excited states. In this paper, we construct solutions to NLS behaving like a sum of $N$ excited states which spread up quickly as time grows (which we call multi-solitons). We also show that if the flow around one of these excited states is linearly unstable, then the multi-soliton is not unique, and is unstable. 

\end{abstract}

\maketitle

\section{Introduction}

\subsection*{Setting of the problem}

We consider the nonlinear Schrödinger equation 
\begin{equation}\label{eq:nls} \tag{NLS}
iu_t+\D u+f(u)=0
\end{equation}
where $u:\R\times\Rd\to\C$ and $f:\C\to\C$ is defined for any $z\in\C$ by
$f(z)=g(|z|^2)z$ with $g\in\mathcal{C}^0([0,+\infty),\R)\cap\mathcal{C}^1((0,+\infty),\R)$.  

Equation \eqref{eq:nls} admits special travelling wave solutions called solitons:
given a frequency $\w_0>0$, an initial phase $\g_0\in\R$, 
initial position and speed $x_0,v_0\in\Rd$ and a solution $\Phi_0\in\hu$ of 
\begin{equation}\label{eq:snls}
-\D \Phi_0+\w_0\Phi_0-f(\Phi_0)=0,
\end{equation}
a \emph{soliton} solution of \eqref{eq:nls} travelling on the line
$x=x_0+v_0t$ is given by 
\begin{equation}\label{eq:defsoliton}
R_{\Phi_0,\w_0,\g_0,v_0,x_0}(t,x):=\Phi_0(x-v_0t-x_0)e^{i(\frac{1}{2}v_0\cdot
  x-\frac{1}{4}|v_0|^2t+\w_0t+\g_0)}. 
\end{equation}

Among solutions of \eqref{eq:snls}, it is common to distinguish between
\emph{ground states}, and \emph{excited states}. A \emph{ground state} (or \emph{least energy solution})
minimizes among all solutions of \eqref{eq:snls} the action $S_0$, defined
for $v \in H^1(\Rd)$ by \label{page:action}
\[
S_0(v):=\frac{1}{2}\nldd{\nabla v}+\frac{\w_0}{2}\nldd{v}-\intrd F(v)dx,
\]
where $F(z):=\int_0^ {|z|}g(s^2)sds$ for all $z\in\C$.
An \emph{excited state} is a solution to \eqref{eq:snls} which is not a
ground state. In general, we shall refer to any solution of \eqref{eq:snls}
as \emph{bound state}. We also mention the existence of a particular type
of excited states, the \emph{vortices}. A vortex is a special solution of
\eqref{eq:snls} which is non-trivially complex-valued, i.e. with a non-zero
angular momentum. Vortices can be constructed following the ansatz
described by Lions in \cite{Li86}. We shall sometimes abuse terminology and
call ground state (resp. excited state) a soliton build with a ground state
(resp. an excited state).

A multi-soliton is a solution of \eqref{eq:nls} built with solitons. More
precisely, let $N\in\N\setminus\{0,1\}$, $\w_1,...,\w_N>0$,
$\g_1,...,\g_N\in\R$, $v_1,...,v_N\in\Rd$, $x_1,...,x_N\in\Rd$ and
$\Phi_1,...,\Phi_N\in\hu$ solutions of \eqref{eq:snls} (with $\w_0$ replaced by $\w_1,...,\w_N$). Set 
\begin{equation}
R_j(t,x):=R_{\Phi_j,\w_j,\g_j,v_j,x_j}(t,x), \quad
R(t,x):=\sum_{j=1}^N R_j(t,x).
\end{equation}
Due to the non-linearity, the function $R$ is not a solution of
\eqref{eq:nls} anymore. What we call \emph{multi-soliton} is a solution $u$
of \eqref{eq:nls} defined on $[T_0,+\infty)$ for some $T_0\in\R$ and such
  that 
\[
\lim_{t\to+\infty}\nhu{u(t)-R(t)}=0.
\]
In this paper, we are concerned with existence, non-uniqueness and
instability of multi-solitons build on excited states, which we will refer
to as excited multi-solitons.

\subsection*{History and known results}

Solitons and multi-solitons play a crucial role in understanding the
dynamics of nonlinear dispersive evolution equations such as Korteweg-de
Vries equations or nonlinear Schrödinger equations (see e.g. \cite{Ta09}
for a general overview). 

To fix ideas, consider the pure-power nonlinearity $f(u) = |u|^{p-1} u$. Equation \eqref{eq:nls} is $L^2$-critical (resp. subcritical, resp. supercritical) if $p= 1 + \frac{4}{d}$ (resp. $p<1+ \frac{4}{d}$, $p>1+ \frac{4}{d}$). The \emph{soliton resolution conjecture} states that, 
at least in the $L^2$-subcritical case, a generic solution will eventually decompose into a sum of \emph{ground state} solitons and
a small radiative term, in some sense we will not try to make precise.
However, this conjecture remains widely open, except when the equation is completely
integrable (like the classical Korteweg-de Vries equation $u_t + u_{xxx} +
uu_x=0$) and explicit solutions are known \cite{La80,Sc86}.

Nevertheless,  multi-solitons \emph{based on ground states} are supposed to be
generic objects for large time; in contrast \emph{excited} multi-solitons are
believed to be singular objects of the flow of \eqref{eq:nls}. However, their
existence shows that a global approach of the large time dynamics must take
care of them.

The first existence result of multi-solitons in a non-integrable setting
was obtained by Merle \cite{Me90} for multi-solitons composed of ground states or excited states for the $L^2$-critical nonlinear Schrödinger equation. For multi-solitons composed only of ground states, the 
$L^2$-subcritical case was treated by Martel and Merle \cite{MaMe06} (see
also Martel \cite{Ma05} for the generalized Korteweg-de Vries equation) and
the $L^2$-supercritical case by Côte, Martel and
Merle \cite{CoMaMe09}. No excited multi-solitons were ever constructed except in the $L^2$-critical case and our result (Theorem
\ref{thm:1}) is the first in that direction: we construct excited
multi-solitons based on excited states which move fast away from one
another.

Study of the dynamics around \emph{ground-states} solitons and
multi-solitons, in particular stability properties, has attracted a lot of
attention since the beginning of the 80's (see e.g.
\cite{BeCa81,CaLi82,GrShSt87,GrShSt90,We83,We85,We86}). The main result states that ground-states solitons are
orbitally stable only in the $L^2$-subcritical case.

So far, little is known about the stability of excited state solitons. All excited states are conjectured to be unstable, regardless of any
assumption on the nonlinearity. For results on instability with
a supercritical nonlinearity, see Grillakis \cite{Gr88} and Jones
\cite{Jo88} in the case of real and radial excited states and Mizumachi for
vortices \cite{Mi05-1,Mi05-2}. Partial results in the $L^2$-subcritical
case are available in the works of Chang, Gustafson, Nakanishi and Tsai
\cite{ChGuNaTs07}, Grillakis \cite{Gr90} and Mizumachi \cite{Mi07}.

Here we show that under a very natural assumption of instability of the
linearized flow around one excited state, the excited multi-soliton is not
unique, and unstable in a strong sense.

\subsection*{Statement of the results}

We make the following assumptions on the nonlinearity (recall that
$f(z)=g(|z|^2)z$ for $z\in\C$). 

\begin{itemize}
\item[(A1)] $g\in\mathcal{C}^0([0,+\infty),\R)\cap\mathcal{C}^1((0,+\infty),\R)$, $g(0)=0$ and $\lim_{s\to 0} sg'(s)=0$.
\item[(A2)] There exist $C>0$ and $1< p<1+\frac{4}{d-2}$ if $d \geq 3$,
  $1<p<+\infty$ if $d=1,2$ such that $|s^2g'(s^2)|\leq Cs^{p-1}$ for
  $s\geq1$. 
\item[(A3)]There exists $s_0 > 0$ such that $F(s_0)>\frac{s_0^2}{2}$.
\end{itemize}

\begin{rmk}
 A typical example of a non-linearity satisfying (A1)-(A3) is given by the power type non-linearity $f(z)=|z|^{p-1}z$ with $1<p<1+\frac{4}{d-2}$ if $d\geq3$, $1<p<+\infty$ if $d=1,2$.
\end{rmk}

Assumptions (A1)-(A3) guarantee that, except in dimension $d=1$ where all
bound states are ground states, there exist ground states and infinitely
many excited states (see
e.g. \cite{BeGaKa83,BeLi83-1,BeLi83-2,Gr90,JoKu86}). In particular, excited states can have arbitrarily large energy and $L^\infty(\Rd)$-norm. Note that every solution of \eqref{eq:snls} is exponentially decaying (see
e.g. \cite{BeSh91}). More precisely, for all $\Phi_0$ solution to
\eqref{eq:snls} we have $ 
e^{\sqrt{\w} |x|}(|\Phi_0|+|\nabla\Phi_0|)\in L^\infty(\Rd)$ for all $\w < \w_0$.

Assumptions (A1)-(A2) ensure  well-posedness in $\hu$ of \eqref{eq:nls},
see e.g. \cite{Ca03} (the equation is then $H^1$-subcritical). In
particular, for any $u_0\in\hu$ there exists a 
unique maximal solution $u$ such that energy, mass and momentum are
conserved. Recall that energy, mass and momentum are defined in the
following way. 
\begin{align*}
E(u)&:=\frac{1}{2}\nldd{\nabla u}-\intrd F(u)dx,\\
M(u)&:=\nldd{u},\\
P(u)&:=\Im\intrd\bar{u}\nabla udx.
\end{align*}
Notice that (A3) makes the equation focusing.

Our first result is the existence of multi-solitons composed of excited
states as soon as the relative speeds $v_j - v_k$ of the solitons are
sufficiently large. 

\begin{thm}\label{thm:1} 
Assume (A1)-(A3).
Let $N\in\N\setminus\{0,1\}$, and for $j=1,...,N$ take $\w_j>0$,
$\g_j\in\R$, $v_j\in\Rd$, $x_j\in\Rd$ and
$\Phi_j\in\hu$ a solution of \eqref{eq:snls} (with $\w_0$ replaced by $\w_j$). Set 
\[
R_j(t,x)=R_{\Phi_j,\w_j,\g_j,v_j,x_j}(t,x):=\Phi_j(x-v_jt-x_j)e^{i(\frac{1}{2}v_j\cdot
  x-\frac{1}{4}|v_j|^2t+\w_jt+\g_j)}.
\]
Let $\w_\star$ and $v_\star$ be given by 
\[
\w_\star:=\frac{1}{2}\min\left\{ \w_j,j=1,...,N \right\},\quad
v_\star:=\frac{1}{9}\min\left\{|v_{j}-v_k|;j,k=1,...,N,j\neq k
\right\}. 
\]
Also introduce
$\alpha:=\sin\left(\frac{\sqrt{\pi}\Gamma(\frac{d-1}{2})}{N^2\Gamma(\frac{d}{2})}\right)$    
(this constant appears naturally in Claim
\ref{cl:choicedirection}).\\
There exists $v_\sharp:=v_\sharp(\Phi_1,...,\Phi_N)>0$ such that if
$v_\star>\alpha^{-1}v_\sharp$ then the following holds.\\ 
There exist $T_0\in\R$ and a solution of \eqref{eq:nls}
$u\in\mathcal{C}([T_0,+\infty),\hu)$ such that for all $t\in[T_0,+\infty)$
we have 
\[ 
\nhu{u(t)-\sum_{j=1}^N R_j(t)}\leq e^{-\alpha\w_\star^{\frac{1}{2}} v_\star t}.
\]
\end{thm}

We now turn to the non-uniqueness and instability of a multi-soliton.

Assume that the flow around one of the $R_j$ is linearly unstable, i.e. has
an eigenvalue off the imaginary axis. As the $R_j$ all play the same role,
we can assume it is $R_1$. 
\begin{itemize}
\item[(A4)] $L = -i\Delta +i\w_1 - idf(\Phi_1)$ has an eigenvalue
  $\lambda \in \m C$ with $\rho := \Re (\lambda) >0$. 
\end{itemize}
This assumption is very natural if one expects $R_1$ to be
unstable. Actually, (A4) holds for any real radial bound state in the $L^2$-supercritical
case (see \cite{Gr88}). For excited states, (A4) is believed to hold for a wide class of
 non-linearities.

Under assumption (A4), we are able to construct a one parameter family of  solutions to \eqref{eq:nls} that converge to the soliton $R_1$ as time goes to infinity, as described in the following Theorem. 

\begin{thm} \label{thm:2}
Take $\w_1>0$,
$\g_1\in\R$, $v_1\in\Rd$, $x_1\in\Rd$ and
$\Phi_1\in\hu$ a solution of \eqref{eq:snls} (with $\w_0$ replaced by $\w_1$). Set 
\[
R_1(t,x)=R_{\Phi_1,\w_1,\g_1,v_1,x_1}(t,x):=\Phi_1(x-v_1t-x_1)e^{i(\frac{1}{2}v_1\cdot
  x-\frac{1}{4}|v_1|^2t+\w_1t+\g_1)}.
\]
Assume $g$ is $\mathcal C^\infty$ and (A1)-(A4) are satisfied. 

There exists a function $Y(t)$ such that $\nhu{Y(t)}\leq Ce^{-\rho t}$ and $e^{\rho t} \| Y(t) \|_{H^1}$ is non-zero and periodic (here $\rho$ is given by (A4) and $Y(t)$ is actually a solution to the linearized flow around $R_1$, see \eqref{def:Y},  \eqref{eq:Ylinflow}). For all $a \in \m
R$, there exist $T_0\in\R$ large enough, a solution $u_a$ to \eqref{eq:nls}
defined on $[T_0,+\infty)$, and a constant $C>0$ such that 
\[ \forall t \ge T_0, \quad \| u_a(t) - R_1(t) - a Y(t) \|_{H^1(\Rd)} \le C
e^{-2\rho t}. \] 
\end{thm}
In particular, Theorem~\ref{thm:2} implies that the soliton $R_1$ is orbitally unstable, as precised in the following corollary.

\begin{cor}\label{cor1}
Under the hypotheses of Theorem~\ref{thm:2}, $R_1$ is orbitally unstable in the following sense. Let $\sigma\geq0$. There exist $\varepsilon>0$, $(T_n)\subset(0,+\infty)$, $(u_{0,n})\subset\hu$ and solutions $(u_n)$ of \eqref{eq:nls} defined on  $[0,T_n]$ with $u_n(0)=u_{0,n}$ such that
\[
\lim_{n\to+\infty}\|{u_{0,n}-R_1(0)}\|_{H^\sigma(\Rd)}=0 \mbox{ and }\inf_{y\in\Rd,\vartheta\in\R}\nld{u_n(T_n)-e^{i\vartheta}\Phi_1(\cdot-y)}\geq\varepsilon\mbox{ for all }n\in\N.
\]
\end{cor}

From Theorem~\ref{thm:2} we infer the existence of a one parameter family of multi-solitons. As a corollary, we obtain non-uniqueness
and instability for high relative speeds multi-solitons. 

\begin{thm} \label{thm:3}
Let $N\in\N\setminus\{0,1\}$, and for $j=1,...,N$ take $\w_j>0$,
$\g_j\in\R$, $v_j\in\Rd$, $x_j\in\Rd$ and
$\Phi_j\in\hu$ a solution of \eqref{eq:snls} (with $\w_0$ replaced by $\w_j$). Set 
\[
R_j(t,x)=R_{\Phi_j,\w_j,\g_j,v_j,x_j}(t,x):=\Phi_j(x-v_jt-x_j)e^{i(\frac{1}{2}v_j\cdot
  x-\frac{1}{4}|v_j|^2t+\w_jt+\g_j)}.
\]
Let $
v_\star:=\frac{1}{9}\min\left\{|v_{j}-v_k|;j,k=1,...,N,j\neq k
\right\}. 
$
Assume $g$ is $\mathcal{C}^\infty$ and (A1)-(A4) are satisfied. 

There exists $v_\natural := v_\natural(\Phi_1, \dots, \Phi_N)>0$ such that if $v_\star>v_\natural$ then the following holds.\\
There exists a function $Y(t)$ such that $\nhu{Y(t)}\leq Ce^{-\rho t}$ and $e^{\rho t} \| Y(t) \|_{H^1}$ is non-zero and periodic (here $\rho$ is given by (A4) and $Y(t)$ is actually a solution to the linearized flow around $R_1$, see \eqref{def:Y},  \eqref{eq:Ylinflow}).
 For all $a \in \m
R$, there exist $T_0\in\R$ large enough, a solution $u_a$ to \eqref{eq:nls}
defined on $[T_0,+\infty)$, and a constant $C>0$ such that 
\[ \forall t \ge T_0, \quad \| u_a(t) - \sum_{j=1}^N R_j(t) - a Y(t) \|_{H^1(\Rd)} \le C
e^{-2\rho t}. \] 
\end{thm}

\begin{rmk}
Notice that, in Theorem \ref{thm:3}, if for $a,b\in\R$ we have $a\neq b$, then $u_a\not\equiv u_b$. Indeed, for $t$ large enough we have
\[
\nhu{u_a(t)-u_b(t)}\geq |a-b|\nhu{Y(t)}-2Ce^{-2\rho t}.
\]
Since $e^{\rho t} \| Y(t) \|_{H^1}$ is non-zero and periodic, this implies that $u_a\not\equiv u_b$ if $a\neq b$.
\end{rmk}

\begin{cor}\label{cor2}
Under the hypotheses of Theorem \ref{thm:3}, the following instability property holds.
Let $\sigma \ge 0$, there exists $\e >0$, such that
for all $n\in\N\setminus\{0\}$ and for all $T \in \m R$ the following holds. There exists
$I_n, J_n \in \m R$, $T \le I_n < J_n$  and a solution $w_n \in  \q C([I_n, J_n], H^1(\m R^d))$ to \eqref{eq:nls} such that
\[ 
\lim_{n\to+\infty}\| w_n(I_n) - R(I_n) \|_{H^\sigma(\m R^d)} =0, \quad \text{and} \quad \inf_{\substack{ y_j\in\Rd, \vartheta_j\in\R,\\j=1,...,N}} \| w_n(J_n) - \sum_{i=1}^N \Phi_j(x-y_j) e^{i (\frac{1}{2}v_j\cdot
  x+\vartheta_j)} \|_{L^2(\m R^d)} \ge \e. 
\]
\end{cor}

\begin{rmk}
The fact that instability holds backward in time (i.e. with $J_n<I_n$) is an easy consequence of Theorem \ref{thm:3}. Hence the difficulty in Corollary  \ref{cor2} is to prove instability forward in time.
\end{rmk}

\begin{rmk}
The classification of multi-solitons is now complete for the generalized Korteweg-de Vries equations (see \cite{Co10,Ma05} and the references therein). In particular, uniqueness holds in the subcritical and critical cases, whereas in the supercritical case the set of multi-solitons consists in a $N$-parameters family.
To the authors knowledge, no uniqueness nor classification result is available yet for multi-solitons of nonlinear Schrödinger equations.
\end{rmk}

\subsection*{Scheme of proofs and comments}

Our strategy for the proof of the existence result (Theorem~\ref{thm:1}) is
inspired from the works \cite{CoMaMe09,MaMe06,Me90}: we take a sequence of time
$T_n\to+\infty$ and a set of final data $u_n(T_n)=R(T_n)$. Our goal is to prove that the solutions $u_n$ to \eqref{eq:nls} backwards in time (which approximate a multi-soliton) exist up to some time $T_0$ independent of $n$, and enjoy uniform $H^1(\m R^d)$ decay estimates on $[T_0,T_n]$. 
A compactness argument then shows that $(u_n)$ converges to a multi-soliton solution to
\eqref{eq:nls} defined on $[T_0,+\infty)$.  

As in \cite{CoMaMe09,MaMe06}, the
uniform backward $\hu$-estimates rely on slow variation
of localized conservation laws as well as coercivity of the Hessian
of the action around each component of the multi-soliton. 
However, this Hessian has negative ``bad  directions'' on which it is not coercive. When dealing with ground states,
these were ruled out either by modulation and conservation of the mass (as in
\cite{MaMe06}) or with the help of explicit knowledge of eigenfunctions of
the operator corresponding to the linearization of \eqref{eq:nls} around a
soliton (as in \cite{CoMaMe09}). In both cases, this could be done
only because of the knowledge of precise spectral properties for ground
states; this does no longer hold when dealing with the more general case of excited states.

Our remark is that the Hessian fails to be $\hu$-coercive only up to a $\ld$-scalar
product with the bad directions. Hence the first step in our analysis is to
find  uniform $\ld$-backward estimates without the help of the
Hessian. This rules out the ``bad directions'' and we can now take advantage of
the coercivity of the Hessian to obtain the $\hu$-estimates.
The main drawback of our approach is that the bootstrap of the $\ld$-estimates requires that the soliton components are
well-separated. Thus we have to work with high-speed solitons. 

To obtain the one parameter family of Theorem~\ref{thm:2}, we rely on a
fixed point argument for smooth functions exponentially convergent (in
time). This is possible because we now assume smoothness on the
non-linearity. The main difficulty is to construct a very good approximate
solution to the multi-soliton. Actually we build such a profile at
arbitrary exponential order. This method is inspired by 
\cite{DuMe07,DuMe08,DuMe09,DuRo} in the case of a single ground state, for the
nonlinear wave or Schrödinger equations. It was also recently developed by
Combet \cite{Co09,Co10} for multi-solitons in the context of the
$L^2$-supercritical generalized Korteweg-de Vries equation. 

However, an important difference in our case is that we consider excited
states, and the linearized flow around them is much less understood than
that around a ground state soliton. For example,
to our knowledge, the exponential decay of eigenfunctions was not known in
general (see \cite{HuLe07} for a partial result). We prove it in
 \ref{ap:decay}, see Proposition~\ref{thm:decayofeigenfunction}. Also, the
unstable eigenvalue has no reason to be real, and this will make the
construction of the profile much more intricate than in the ground state
soliton case. This is the purpose of Proposition~\ref{prop:profile}. Once
the approximation profile is derived, the proofs of Theorem~\ref{thm:2} and
\ref{thm:3} follow from a fixed point argument around the profile.

The paper is organized as follows. In Section 2 we prove Theorem
\ref{thm:1}. Section 3 is devoted to the proofs of Theorems~\ref{thm:2}
and~\ref{thm:3}. In \ref{ap:decay} we prove the exponential decay of eigenfunctions for matrix Schrödinger operators and  in \ref{ap:cor} we prove Corollaries \ref{cor1} and \ref{cor2}. 

\noindent\textbf{Acknowledgement.} The authors are grateful to the unknown referee for valuable comments and suggestions.

\section{Existence}

In this section, we assume (A1)-(A3) and suppose we are given $N\in\N\setminus\{0,1\}$, and for $j=1,...,N$, $\w_j>0$,
$\g_j\in\R$, $v_j\in\Rd$, $x_j\in\Rd$ and
$\Phi_j\in\hu$ a solution of \eqref{eq:snls} (with $\w_0$ replaced by $\w_j$). Recall that 
\begin{gather*}
R_j(t,x)=\Phi_j(x-v_jt-x_j)e^{i(\frac{1}{2}v_j\cdot
  x-\frac{1}{4}|v_j|^2t+\w_jt+\g_j)},\\
\w_\star=\frac{1}{2}\min\left\{ \w_j,j=1,...,N \right\},\quad
v_\star=\frac{1}{9}\min\left\{|v_{j}-v_k|;j,k=1,...,N,j\neq k
\right\},
\end{gather*}
and $\alpha:=\sin\left(\frac{\sqrt{\pi}\Gamma(\frac{d-1}{2})}{N^2\Gamma(\frac{d}{2})}\right)$.

\subsection{Approximate solutions and convergence toward a multi-soliton}\label{sec:approximate}

Let $(T_n)_{n \ge 1} \subset\R$ be an increasing sequence of time such that $T_n\to+\infty$ and $(u_n)$ be solutions to \eqref{eq:nls} such that $u_n(T_n)=R(T_n)$. We call $u_n$ an \emph{approximate multi-soliton}.

The proof of Theorem~\ref{thm:1} relies on the following proposition.
\begin{prop}[Uniform estimates]\label{prop:estimate}
There exists $v_\sharp:=v_\sharp(\Phi_1,...,\Phi_N)>0$ such that if $v_\star>\alpha^{-1}v_\sharp$ then the following holds. There exist $n_0\in\N,$ $T_0>0$ such that for all $n\geq n_0$ every approximate multi-soliton $u_n$ is defined on $[T_0,T_n]$ and for all $t\in[T_0,T_n]$ we have
\begin{equation}\label{eq:prop:estimate}
\nhu{u_n(t)-R(t)} \leq e^{-\alpha\w_\star^{\frac{1}{2}} v_\star t}.
\end{equation}
\end{prop}

In this section, assuming Proposition~\ref{prop:estimate}, we prove Theorem 1 by establishing the convergence of the approximate multi-solitons $u_n$ to a multi-soliton $u$ existing on $[T_0,+\infty)$. Our proof follows the same line as in \cite{CoMaMe09,MaMe06}.

From now on and in the rest of section~\ref{sec:approximate} we assume that $v_\star>\alpha^{-1}v_\sharp$, where $v_\sharp$ is given by Proposition~\ref{prop:estimate}.

Since the approximate multi-solitons $u_n$ are constructed by solving \eqref{eq:nls} backward in time, to prove Theorem~\ref{thm:1} we first need to make sure that the initial data $u_n(T_0)$ converge to some initial datum $u_0$.

\begin{lem}\label{lem:compactness}
There exists $u_0\in\hu$ such that, possibly for a subsequence only, $u_n(T_0)\to u_0$ strongly in $H^s(\Rd)$ as $n\to+\infty$ for any $s\in[0,1)$. 
\end{lem}

Lemma~\ref{lem:compactness} is a consequence of the following claim.

\begin{claim}[$L^2(\Rd)$-compactness]\label{cl:compactness}
Take $\delta>0$. There exists $r_\delta>0$ such that for all $n$ large enough we have
\begin{equation}\label{eq:lemcvclaim}
\int_{|x|>r_\delta}|u_n(T_0)|^2dx\leq \delta.
\end{equation}
\end{claim}

\begin{proof}
Let $n$ be large enough so that the conclusions of Proposition~\ref{prop:estimate} hold. Let $T_\delta$ be such that $e^{-\alpha\w_\star^{\frac{1}{2}}  v_\star T_\delta}\leq \sqrt{\frac{\delta}{4}}$. Then, by Proposition~\ref{prop:estimate}, we have
\begin{equation}\label{eq:lemcv1}
\nhu{u_n(T_\delta)-R(T_\delta)}\leq \sqrt{\frac{\delta}{4}}.
\end{equation}
Let $\rho_\delta$ be such that 
\begin{equation}\label{eq:lemcv2}
\int_{|x|>\rho_\delta}|R(T_\delta)|^2dx<\frac{\delta}{4}.
\end{equation}
From \eqref{eq:lemcv1}-\eqref{eq:lemcv2} we infer
\begin{equation}\label{eq:lemcv3}
\int_{|x|>\rho_\delta}|u_n(T_\delta)|^2dx<\frac{\delta}{2}.
\end{equation}
We define a $\mathcal{C}^1$ cut-off function $\tau:\R\to\R$ such that $\tau(s)=0$ if $s\leq 0$, $\tau(s)=1$ if $s\geq 1$, $\tau(s)\in[0,1]$ and $|\tau'(s)|\leq 2$ if $s\in[0,1]$. Let $\kappa_\delta$ to be determined later and consider
\[
\Upsilon(t):=\intrd|u_n(t)|^2\tau\left(\frac{|x|-\rho_\delta}{\kappa_\delta}\right)dx. 
\]
To obtain \eqref{eq:lemcvclaim} from \eqref{eq:lemcv3} we need to establish a link between $\Upsilon(T_0)$ and $\Upsilon(T_\delta)$. Differentiating in time, we obtain after simple calculations (see e.g. \cite[Claim 2]{MaMe06})
\[
\Upsilon'(t)=\frac{2}{\kappa_\delta}\Im\intrd\bar{u}_n\nabla u_n\cdot\frac{x}{|x|}\tau'\left(\frac{|x|-\rho_\delta}{\kappa_\delta}\right)dx.
\]
Since $\nhu{u_n(t)}$ is bounded independently of $n$ and $t$, there exists 
\[
C_0:=\sup_{n\in\N}\sup_{t\in[T_0,T_n]}\nhud{u_n(t)}>0
\]
 such that 
\[
|\Upsilon'(t)|\leq \frac{2C_0}{\kappa_\delta}.
\]
Choose $\kappa_\delta$ such that $\frac{2C_0}{\kappa_\delta}T_\delta<\frac{\delta}{2}$. Then, by integrating between $T_0$ and $T_\delta$ we obtain
\begin{equation}\label{eq:lemcv4}
\Upsilon(T_0)-\Upsilon(T_\delta)\leq \frac{\delta}{2}.
\end{equation}
From \eqref{eq:lemcv3} we infer that 
\[
\Upsilon(T_\delta)=\intrd|u_n(T_\delta)|^2\tau\left(\frac{|x|-\rho_\delta}{\kappa_\delta}\right)dx\leq\int_{|x|>\rho_\delta}|u_n(T_\delta)|^2dx\leq \frac{\delta}{2}.
\]
Combining with \eqref{eq:lemcv4} we obtain
\[
\Upsilon(T_0)\leq \delta.
\]
Now set $r_\delta:=\kappa_\delta+\rho_\delta$. Then from the definition of $\tau$ it is easy to see that 
\[
\int_{|x|>r_\delta}|u_n(T_0)|^2dx\leq \Upsilon(T_0)\leq \delta,
\]
which proves the claim.
\end{proof}

\begin{proof}[Proof of Lemma~\ref{lem:compactness}]
Since $u_n(T_0)$ is bounded in $\hu$, there exists $u_0\in\hu$ such that up to a subsequence $u_n(T_0)\rightharpoonup u_0$ weakly in $\hu$. Hence, $u_n(T_0)\to u_0$ strongly in $L^2_{\mathrm{loc}}(\Rd)$ and actually strongly in $\ld$ by Claim~\ref{cl:compactness}. By interpolation we get the desired conclusion.
\end{proof}

\begin{proof}[Proof of Theorem~\ref{thm:1}]
Let $u_0$ be given by Lemma~\ref{lem:compactness} and let $u\in\mathcal{C}([T_0,T^\star),\hu)$ be the corresponding maximal solution of \eqref{eq:nls}. By (A1)-(A2), there exists $0<\s<1$ such that $1< p<1+\frac{4}{d-2\s}$ and
\[
|f(z_1)-f(z_2)|\leq C(1+|z_1|^{p-1}+|z_2|^{p-1})|z_1-z_2|\mbox{ for all }z_1,z_2\in\C.
\]
This implies that the Cauchy problem for \eqref{eq:nls} is well-posed in $H^\s(\Rd)$ (see \cite{Ca03,CaWe90}).
Combined with Lemma~\ref{lem:compactness} this implies that $u_n(t)\to u(t)$ strongly in $H^\s(\Rd)$ for any $t\in[T_0,T^\star)$. By boundedness of $u_n(t)$ in $\hu$, we also have $u_n(t)\rightharpoonup u(t)$ weakly in $\hu$ for any $t\in[T_0,T^\star)$. By Proposition~\ref{prop:estimate}, for any $t\in[T_0,T^\star)$ we have
\begin{equation}\label{eq:proofthm1}
\nhu{u(t)-R(t)}\leq \liminf_{n\to+\infty}\nhu{u_n(t)-R(t)}\leq e^{-\alpha\w_\star^{\frac{1}{2}}  v_\star t}.
\end{equation}
In particular, since $R(t)$ is bounded in $\hu$ there exists $C>0$ such that for any $t\in[T_0,T^\star)$ we have
\begin{equation}\label{eq:proofthm2}
\nhu{u(t)}\leq e^{-\alpha\w_\star^{\frac{1}{2}}  v_\star t}+\nhu{R(t)}\leq C.
\end{equation}
Recall that, by the blow up alternative (see e.g. \cite{Ca03}), either $T^\star=+\infty$ or $T^\star<+\infty$ and $\lim_{t\to T^\star}\nhu{u}=+\infty$. Therefore \eqref{eq:proofthm2} implies that $T^\star =+\infty$. From \eqref{eq:proofthm1} we infer that for all $t\in[T_0,+\infty)$ we have 
\[
\nhu{u(t)-R(t)}\leq e^{-\alpha\w_\star^{\frac{1}{2}}  v_\star t}.
\]
This concludes the proof.
\end{proof}

\subsection{Uniform backward estimates}

This section is devoted to the proof of Proposition~\ref{prop:estimate}. This proof relies on a bootstrap argument. Indeed, from the definition of the final datum $u_n(T_n)$ and continuity of $u_n$ in time, it follows that \eqref{eq:prop:estimate} holds on an interval $[t^\dagger,T_n]$ for $t^{\dagger}$ close enough to $T_n$. Then the following Proposition~\ref{prop:estimate2} shows that we can actually improve to a better estimate, hence leaving enough room to extend the interval on which the original estimate holds.  

\begin{prop}\label{prop:estimate2}
There exists $v_\sharp:=v_\sharp(\Phi_1,...,\Phi_N)>0$ such that if $v_\star>\alpha^{-1}v_\sharp$ then the following holds. There exist $n_0\in\N,$ $T_0>0$ such that for all $n\geq n_0$ every approximate multi-soliton $u_n$ is defined on $[T_0,T_n]$. Let $t^{\dagger}\in[T_0,T_n]$ and $n\geq n_0$. If for all $t\in[t^{\dagger},T_n]$ we have  
\begin{equation}\label{eq:H1est}
\nhu{u_n(t)-R(t)}\leq e^{-\alpha\w_\star^{\frac{1}{2}}  v_\star t}
\end{equation}
then for all $t\in[t^{\dagger},T_n]$ we have 
\begin{equation}\label{eq:H1est-bootstrap}
\nhu{u_n(t)-R(t)}\leq \frac{1}{2} e^{-\alpha\w_\star^{\frac{1}{2}}  v_\star t}.
\end{equation}
\end{prop}

Before proving Proposition~\ref{prop:estimate2}, we indicate precisely how it is used to obtain Proposition~\ref{prop:estimate}.

\begin{proof}[Proof of Proposition~\ref{prop:estimate}]
Let $T_0$, $n_0$ and $v_\sharp$ be given by Proposition~\ref{prop:estimate2}, assume $v_\star>\alpha^{-1}v_\sharp$, and let $n\geq n_0$.
Since $u_n(T_n)=R(T_n)$ and $u_n$ is continuous in $\hu$, for $t$ close enough to $T_n$ we have
\begin{equation}\label{eq:prop:estimate1}
\nhu{u_n(t)-R(t)}\leq e^{-\alpha\w_\star^{\frac{1}{2}}  v_\star t}.
\end{equation}
Let $t^\dagger$ be the minimal time such that \eqref{eq:prop:estimate1} holds:
\[
t^\dagger:=\min\{\tau\in[T_0,T_n];\eqref{eq:prop:estimate1}\mbox{ holds for all }t\in[\tau,T_n] \}.
\]
We prove by contradiction that $t^\dagger=T_0$. Indeed, assume that $t^\dagger>T_0$. Then
\[
\nhu{u_n(t^\dagger)-R(t^\dagger)}\leq e^{-\alpha\w_\star^{\frac{1}{2}}  v_\star t^\dagger}
\]
and by Proposition~\ref{prop:estimate2} we can improve this estimate in
\[
\nhu{u_n(t^\dagger)-R(t^\dagger)}\leq \frac{1}{2} e^{-\alpha\w_\star^{\frac{1}{2}}  v_\star t^\dagger}.
\]
Hence, by continuity of $u_n(t)$ in $\hu$, there exists $T_0\leq t^\ddagger<t^\dagger$ such that \eqref{eq:prop:estimate1} holds for all $t\in[t^\ddagger,t^\dagger]$. This contradicts the minimality of $t^\dagger$ and finishes the proof.
\end{proof}

The proof of Proposition~\ref{prop:estimate2} is done in two steps. First, assuming \eqref{eq:H1est} we prove that we can control the $L^2(\Rd)$-norm of $(u_n-R)$. To obtain the full control on the $\hu$-norm of $(u_n-R)$ as in \eqref{eq:H1est-bootstrap} we use the linearization of an action-like functional. This linearization is coercive (i.e. controls the $\hu$-norm) up to a finite number of non-positive directions that can all be controlled due to the $\ld$-estimate.

Let $T_0>0$ large enough and fix $n\in\N$ such that $T_n>T_0$. For notational convenience, the dependency on $n$ is understood for $u$ and we drop the subscript $n$. Set $v:=u-R$. Let $t^{\dagger}\in[T_0,T_n]$ and assume that for all $t\in[t^{\dagger},T_n]$ we have
\[
\nhu{v(t)}\leq e^{-\alpha\w_\star^{\frac{1}{2}}  v_\star t}.
\]

\subsubsection*{\texorpdfstring{Step 1: $\ld$-control}{Step 1: L2-control}}

\begin{lem}\label{lem:step1}
For all $K>0$ and $m\in\N\setminus\{0\}$ there exists $v_\sharp=v_\sharp(K,m,\Phi_1,...,\Phi_N)>0$ such that if $v_\star>\alpha^{-1}v_\sharp$ then for all $t\in[t^{\dagger},T_n]$ we have 
\begin{equation*}
\nld{v(t)}\leq \frac{1}{\sqrt{2mK}} e^{-\alpha\w_\star^{\frac{1}{2}}  v_\star t}.
\end{equation*}
\end{lem}
Notice that the reason why we introduce such $K$ and $m$ will appear later in the proof.

\begin{proof}
First note that by identifying $\C$ to $\R^2$ and viewing $f: \m R^2 \to \m R^2$ we can consider
\[ 
df(z).w = g(|z|^2)w + 2 \Re(z \bar w) g'(|z|^2) z.
\]

The function $v$ satisfies
\[
iv_t+\mathcal{L}v+\mathcal{N}(v)=0,
\]
where
\[
\mathcal{L}v:=\D v+df(R).v 
\]
and the remaining nonlinear term $\mathcal{N}(v)$ verifies
\[
|\psld{i\mathcal{N}(v)}{v}|\leq \eta(\nhu{v})\nhud{v},
\]
where $\eta$ is a decreasing function satisfying $\eta(s)\to 0$ when $s\to 0$.
Take any $t\in[t^{\dagger},T_n]$. We have
\[
\frac{1}{2}\frac{d}{dt}\nldd{v}=\psld{v_t}{v}=\psld{i\mathcal{L}v}{v}+\psld{i\mathcal{N}(v)}{v}.
\]
We have
\begin{align*}
\MoveEqLeft
\psld{i\mathcal{L}v}{v}\\
&=\Re\intrd i\left(\D v+df(R).v)\right)\bar{v}dx,\\
&=\Re\intrd i\left(\D v+g(|R|^2)v+2g'(|R|^2)\Re(R\bar{v})R\right)\bar{v}dx,\\
&=\Re\intrd i(-|\nabla v|^2+g(|R|^2)|v|^2+2g'(|R|^2)\Re(R\bar{v})R\bar{v})dx,\\
&=-\intrd 2g'(|R|^2)\Re(R\bar{v})\Im(R\bar{v}) dx.
\end{align*}
Therefore,
\begin{align*}
|\psld{i\mathcal{L}v}{v}|&\leq\intrd 2|g'(|R|^2)||R|^2|v|^2dx.\\
&\leq \frac{C_\mathcal{L}}{2} \nldd{v},
\end{align*}
where this last constant $C_\mathcal{L}$ depends only on $g$ and $\|R\|_{L^\infty(\Rd)}$. 
By the bootstrap assumption on $v$, this implies
\[
|\psld{i\mathcal{L}v}{v}|\leq \frac{C_\mathcal{L}}{2}e^{-2\alpha\w_\star^{\frac{1}{2}}  v_\star t}.
\]
In addition, it is easy to see that 
\[
|\psld{i\mathcal{N}(v)}{v}|\leq \eta(\nhu{v})\nhud{v}\leq \eta(e^{-\alpha\w_\star^{\frac{1}{2}}  v_\star t})e^{-2\alpha\w_\star^{\frac{1}{2}}  v_\star t}.
\]
In short, if $T_0$ is large enough so that $\eta(e^{-\alpha\w_\star^{\frac{1}{2}}  v_\star t})\leq \frac{C_\mathcal{L}}{2}$, we have obtained that
\[
\left|\frac{d}{dt}\nldd{v}\right|\leq 2C_\mathcal{L}e^{-2\alpha\w_\star^{\frac{1}{2}}  v_\star t}.
\]
Therefore, by integration between $t$ and $T_n$ we get 
\begin{equation}\label{eq:lemstep1-2}
\nldd{v(t)}-\nldd{v(T_n)}\leq \frac{C_\mathcal{L}}{\alpha\w_\star^{\frac{1}{2}}  v_\star }(e^{-2\alpha\w_\star^{\frac{1}{2}}  v_\star t}-e^{-2\alpha\w_\star^{\frac{1}{2}}  v_\star T_n}).
\end{equation}
Now, we take $v_\sharp$ such that
\[
\frac{C_\mathcal{L}}{\w_\star^{\frac{1}{2}}  v_\sharp}<\frac{1}{2mK}.
\]
If $v_\star>\alpha^{-1}v_\sharp$ and since $v(T_n)=0$ we get from \eqref{eq:lemstep1-2} that
\begin{equation*}
\nld{v(t)}\leq \frac{1}{\sqrt{2mK}} e^{-\alpha\w_\star^{\frac{1}{2}}  v_\star t},
\end{equation*}
which is the desired conclusion.
\end{proof}

\subsubsection*{\texorpdfstring{Step 2: $\hu$-control}{Step 2: H1-control}}

The idea of the second step of the proof of Proposition~\ref{prop:estimate2} is reminiscent of the technique used to prove stability for a single soliton in the subcritical case (see e.g. \cite{GrShSt87,GrShSt90,Le09,We85,We86}). Indeed, it is well-known that the linearization of the action functional $S_0$ (see the definition of $S_0$ p.~\ref{page:action}), whose critical points are the solutions of \eqref{eq:snls}, is coercive  on a subspace of $\hu$ of finite codimension in $\ld$. At large time, the components of the multi-soliton are well-separated and thus it is possible to localize the analysis around each soliton to gain an $\hu$-local control, up to a space of finite dimension in $\ld$. But due to Lemma~\ref{lem:step1} we are able to control the remaining $\ld$-directions, hence to close the proof. The idea of looking at localized versions of the invariants of \eqref{eq:nls} was introduced in \cite{Me90} and later developed in \cite{CoMaMe09,Ma05,MaMe06,MaMeTs06}. We shall therefore be sketchy in the proofs, highlighting only the main differences with the previous works.

We start with the case of a single soliton.
\begin{lem}[Coercivity for a soliton]\label{lem:coercivity-H-0}
Let $\w_0>0$, $\g_0\in\R$, $x_0,v_0\in\Rd$ and a solution $\Phi_0\in\hu$ of \eqref{eq:snls}. Then there exist $K_0=K_0(\Phi_0)>0$, $\nu_0\in\N\setminus\{0\}$ and $\tilde X_0^1,...,\tilde X_0^{\nu_0}\in\ld$ such that for $k=1,...,\nu_0$ we have $\nld{\tilde X^k_0}=1$ and for any $w\in\hu$ we have
\[
\nhud{w}\leq K_0 H_0(t,w)+K_0\sum_{k=1}^{\nu_0}\psld{w}{ X_0^k(t)}^2\mbox{ for all }t\in\R,
\]
where 
\begin{align*}
X^k_0(t):= & e^{i(\frac{1}{2}v_0\cdot x-\frac{1}{4}|v_0|^2t+\w_0t+\g_0)}\tilde X^k_0(x-v_0t-x_0),\\
H_0(t,w):= &\nldd{\nabla w}+\left(\w_0+\frac{|v_0|^2}{4}\right)\nldd{w}-v_0\cdot\Im\intrd\bar{w}\nabla wdx\\
&-\intrd\left(g(|R_0|^2)|w|^2+2g'(|R_0|^2)\Re(R_0\bar{w})^2\right)dx,
\end{align*}
and $ R_0 (t,x)$ is the soliton given by \eqref{eq:defsoliton}.
\end{lem}

Lemma~\ref{lem:coercivity-H-0} follows from standard arguments. We included a proof in \ref{ap:coercivity} for the reader's convenience.

We introduce now the localization procedure around each component of the multi-soliton.

We begin by the selection of a particular direction of propagation.

\begin{claim}\label{cl:choicedirection}
Let $0<\alpha<\sin\left(\frac{\sqrt{\pi}\Gamma(\frac{d-1}{2})}{N(N-1)\Gamma(\frac{d}{2})}\right)$. Then there exists an orthonormal basis $(e_1,...,e_d)$ of $\Rd$ such that for all $j,k=1,...,N$,  we have
\[
|(v_j-v_k,e_1)_{\Rd}|\geq \alpha|v_j-v_k|.
\]
\end{claim}

\begin{proof}
For $j\neq k$, set $v_{jk}:=\frac{v_j-v_k}{|v_j-v_k|}$. The claim will be proved if we show that the measure of the set 
\[
\Lambda:=\bigcup_{\stackrel{j,k=1,...,N}{j\neq k}}\{w\in \mathbb{S}^{d-1},|(v_{jk},w)_{\Rd}|\leq \alpha \}
\]
is smaller than the measure of the surface of the unit sphere $\mathbb{S}^{d-1}$.

Take $j,k=1,...,N$; $j\neq k$. Without loss of generality, assume that $v_{jk}=(1,0,...,0)$. Take $w\in\mathbb{S}^{d-1} $ and let $(\theta_1,...,\theta_{d-1})$ be the spherical coordinates of $w$. Then we have 
\[
(v_{jk},w)_{\Rd}=\cos{\theta_1}.
\]
Therefore, after easy calculations we get
\[
\mu(\{w\in \mathbb{S}^{d-1},|(v_{jk},w)_{\Rd}|\leq \alpha \})\leq 2\arcsin(\alpha)
\frac{\pi^{\frac{d-1}{2}}}{\Gamma(\frac{d-1}{2})} 
\]
where $\mu$ is the Lebesgues measure on $\mathbb{S}^{d-1}$ and $\frac{\pi^{\frac{d-1}{2}}}{\Gamma(\frac{d-1}{2})}$ is the area of the $(d-2)$-unit sphere. By subadditivity of the measure this leads to 
\[
\mu(\Lambda)\leq N(N-1)\arcsin(\alpha) \frac{\pi^{\frac{d-1}{2}}}{\Gamma(\frac{d-1}{2})}.
\]
Now, remember that
\[
0<\alpha<\sin\left(\frac{\sqrt{\pi}\Gamma(\frac{d-1}{2})}{N(N-1)\Gamma(\frac{d}{2})}\right).
\]
This implies 
\[
\mu(\Lambda)\leq N(N-1)\arcsin(\alpha) \frac{\pi^{\frac{d-1}{2}}}{\Gamma(\frac{d-1}{2})}<\frac{\pi^{\frac{d}{2}}}{\Gamma(\frac{d}{2})}=\mu(\mathbb{S}^{d-1}).
\]
Therefore $\mu(\mathbb{S}^{d-1}\setminus \Lambda)>0$ and we can pick up $e_1\in\mathbb{S}^{d-1}$ such that for all $j,k=1,...,N$,  we have
\[
|(v_j-v_k,e_1)_{\Rd}|\geq \alpha |v_j-v_k|.
\]
Completing $e_1$ into an orthonormal basis $(e_1,...,e_d)$ of $\Rd$ finishes the proof.
\end{proof}

By invariance of \eqref{eq:nls} with respect to orthonormal transformations we can assume without loss of generality that the basis $(e_1,...,e_d)$ is the canonical basis of $\Rd$. Up to a changes of indices, we can also assume that $v^1_1<...<v_N^1$ where the exponent $1$ in $v^1_{j}$ denote the first coordinate of $v_j=(v^1_j,...,v^d_j)$. 

Let $\psi:\R\to\R$ be a $\mathcal{C}^\infty$ cut-off function such that 
$\psi(s)=0$ for $s<-1$, $\psi(s)\in[0,1]$ if $s\in[-1,1]$ and $\psi(s)=1$ for $s>1$. We define 
\begin{gather*}
m_j:=\frac{1}{2}(v^1_{j-1}+v^1_{j})\mbox{ for }j=2,...,N,\\
\psi_1(t,x):=1,\;
\psi_j(t,x):=\psi(\frac{1}{\sqrt{t}}(x^1-m_jt))\mbox{ for }j=2,...,N.
\end{gather*}
Then we can define
\[
\phi_j=\psi_j-\psi_{j+1}\mbox{ for }j=1,...,N-1,\;\phi_N:=\psi_N.
\]
We introduce localized versions of the energy, charge and momentum. For $j=1,...,N$ we define
\begin{gather*}
E_j(t,w):=\frac{1}{2}\intrd|\nabla w|^2\phi_jdx-\intrd F(w)\phi_jdx,\;\\
M_j(t,w):=\intrd|w|^2\phi_jdx,\;
P_j(t,w):=\Im\intrd (\nabla w) \bar{w}\phi_jdx.
\end{gather*}
We denote by $S_j$ a localized action defined for $w\in\hu$ by
\[
S_j(t,w):=E_j(t,w)+\frac{1}{2}\left(\w_j+\frac{|v_j|^2}{4}\right)M_j(t,w)-\frac{1}{2}v_j\cdot P_j(t,w)
\]
and by $H_j$ a localized linearized defined for $w\in\hu$ by
\begin{multline*}
H_j(t,w):= \intrd|\nabla w|^2\phi_jdx-\intrd\left(g(|R_j|^2)|w|^2+2g'(|R_j|^2)\Re(R_j\bar{w})^2\right)\phi_jdx\\
+\left(\w_j+\frac{|v_j|^2}{4}\right)\intrd|w|^2\phi_jdx-v_j\cdot\Im\intrd\bar{w}\nabla w\phi_jdx.
\end{multline*}
We define an action-like functional for multi-solitons 
\[
\mathcal{S}(t,w):=\sum_{j=1}^NS_j(t,w)
\]
and a corresponding linearized
\[
\mathcal{H}(t,w):=\sum_{j=1}^N H_j(t,w).
\]

We have the following coercivity property on $\mathcal{H}$.

\begin{lem}[Coercivity for the multi-soliton]\label{lem:coercivity-global}
There exists $K=K(\Phi_1,...,\Phi_N)>0$ such that for all $t$ large enough and for all $w\in\hu$ we have
\[
\nhud{w}\leq K\mathcal{H}(t,w)+K\sum_{j=1}^N\sum_{l=1}^{\nu_j}\psld{w}{ X_j^l(t)}^2,
\]
where $(\nu_j)$, $( X_j^l)$ are given for each $R_j$ by Lemma~\ref{lem:coercivity-H-0}.
\end{lem}

\begin{proof}
It is a consequence of Lemma~\ref{lem:coercivity-H-0}  (see \cite[Lemma 4.1]{MaMeTs06}).
\end{proof}

\begin{lem}\label{lem:loc-S}
The following equality holds
\[
S_j(t,u(t))=S_j(t,R_j)+H_j(t,v)+O(e^{-3\alpha\w_\star^{\frac{1}{2}}  v_\star t})+o(\nhud{v}).
\]
\end{lem}

The proof relies on the following claim.

\begin{claim}\label{cl:exponentialestimatesonRj}
For all $x\in\Rd$ and $j,k=1,...,N$ the following inequalities holds.
\begin{gather*}
(|R_k(t,x)|+|\nabla R_k(t,x)|)\phi_j(t,x)\leq Ce^{-2\alpha\w_\star^{\frac{1}{2}} v_\star t}e^{-\frac{\w_\star^{\frac{1}{2}}}{2}|x-v_kt-x_k|}\mbox{ for }j\neq k,\\
(|R_j(t,x)|+|\nabla R_j(t,x)|)(1-\phi_j(t,x))\leq Ce^{-2\alpha\w_\star^{\frac{1}{2}} v_\star t}e^{-\frac{\w_\star^{\frac{1}{2}}}{2}
|x-v_jt-x_j|}.
\end{gather*}
\end{claim}

\begin{proof}
The claim follows immediately from the support properties of $\phi_j$, the definitions of $\w_\star$ and $v_\star$ and exponential decay of $\Phi_j$.
\end{proof}

\begin{proof}[Proof of Lemma~\ref{lem:loc-S}]
The proof is done by writing $u(t)= R(t) +v(t)$ and expanding in the definition of $S
_j$. We start with the terms of order $0$ in $v$. By Claim~\ref{cl:exponentialestimatesonRj} we have
\begin{equation}\label{eq:Sjest-begin}
S_j(t,R)= S_j(t,R_j)+O(e^{-4\alpha\w_\star^{\frac{1}{2}}  v_\star t}).
\end{equation}
We now look at the terms of order $1$ in $v$. Still by Claim~\ref{cl:exponentialestimatesonRj}, taking in addition into account that $\nhu{v}=O(e^{-\alpha\w_\star^{\frac{1}{2}} v_\star t})$ and remembering the equation solved by $R_j$ (see \eqref{eq:eqsolevedbyR}) we obtain,
\begin{align}
\dual{S_j'(t,R)}{v}&= \dual{S'_j(t,R_j)}{v}+O(e^{-3\alpha\w_\star^{\frac{1}{2}}  v_\star t})=O(e^{-3\alpha\w_\star^{\frac{1}{2}}  v_\star t}),\\
\dual{S_j''(t,R)v}{v}&=H_j(t,v)+O(e^{-3\alpha\w_\star^{\frac{1}{2}}  v_\star t})+o(\nhud{v})\label{eq:Sjest-end}.
\end{align}
Gathering \eqref{eq:Sjest-begin}-\eqref{eq:Sjest-end} we obtain the following expansion
\[
S_j(t,u(t))=S_j(t,R_j)+H_j(t,v)+O(e^{-3\alpha\w_\star^{\frac{1}{2}}  v_\star t})+o(\nhud{v}),
\]
which concludes the proof.
\end{proof}

We can now write a Taylor-like expansion for $\mathcal{S}$.
\begin{lem}\label{lem:S-eexpansion}
We have 
\[
\mathcal{S}(t,u)-\mathcal{S}(t,R)=\mathcal{H}(t,v)+o(\nhud{v})+O(e^{-3\alpha\w_\star^{\frac{1}{2}}  v_\star  t}).
\]
\end{lem}

\begin{proof}
In view of Lemma~\ref{lem:loc-S} all we need to prove is 
\[
\mathcal{S}(t,R)=\sum_{j=1}^NS_j(t,R_j)+O(e^{-3\alpha\w_\star^{\frac{1}{2}}  v_\star  t}),
\]
which follows immediately from Claim~\ref{cl:exponentialestimatesonRj}.
\end{proof}

\begin{lem}\label{lem:estimate-dS}
The following estimate holds.
\[
\left| \frac{\partial \mathcal{S}(t,u(t))}{\partial t}\right|\leq \frac{C}{\sqrt{t}}e^{-2\alpha\w_\star^{\frac{1}{2}}  v_\star t}.
\]
\end{lem}

\begin{proof}
We remark that 
\[
\mathcal{S}(t,w)=E(w)+\sum_{j=1}^N\left(\frac{1}{2}\left(\w_j+\frac{|v_j|^2}{4}\right)M_j(t,w)-\frac{1}{2}v_j\cdot P_j(t,w)\right).
\]
Since the energy $E$ is conserved by the flow of \eqref{eq:nls}, to estimate the variations of $\mathcal{S}(t,u(t))$ we only have to study the variations of the localized masses $M_j(t,u(t))$ and momentums $P_j(t,u(t))$.
Take any $j=2,...,N$. We have
\begin{align}\label{eq:massderivation}
\begin{split}
\MoveEqLeft \frac{1}{2}\frac{\partial}{\partial t}\intrd|u(t)|^2\psi_j(t,x)dx\\&=\frac{1}{\sqrt{t}}\intrd \left(\Im(\bar{u}\partial_{1}u)
- |u|^2\frac{x^1+m_jt}{4t}\right)\psi'(\frac{1}{\sqrt{t}}(x^1-m_jt))dx.
\end{split}
\end{align}
Define $I_j:=[m_jt-\sqrt{t},m_jt+\sqrt{t}]\times \R^{d-1}$. From \eqref{eq:massderivation} and the support properties of $\psi$ we obtain
\[
\left|\frac{\partial}{\partial t}\intrd|u(t)|^2\psi_j(t,x)dx\right|\leq \frac{C}{\sqrt{t}}\int_{I_j}|\nabla u|^2+|u|^2dx.
\]
Similarly, for the first component of $P_j$ we have
\begin{multline}\label{eq:momentum1derivation}
\frac{1}{2}\frac{\partial}{\partial t}\intrd \bar{u}\partial_{1}u \psi_jdx=\\
\frac{1}{\sqrt{t}}\intrd \left(
|\partial_{1}u|^2-g(|u|^2)|u|^2+F(u)-\bar{u}\partial_{1}u\frac{x^1+m_jt}{2t}\right)\psi'(\frac{1}{\sqrt{t}}(x^1-m_jt))\\
-\frac{1}{4t}|u|^2\psi'''(\frac{1}{\sqrt{t}}(x^1-m_jt)dx.
\end{multline}
Combining \eqref{eq:momentum1derivation} with the support properties of $\psi$ and (A1)-(A2) we obtain
\[
\left|\frac{\partial}{\partial t}\intrd \bar{u}\partial_{1}u \psi_jdx\right|\leq \frac{C}{\sqrt{t}}\left(\int_{I_j}|\nabla u|^2+|u|^2dx+\left(\int_{I_j}|\nabla u|^2+|u|^2dx\right)^{\frac{p+1}{2}}\right).
\]
Similar arguments lead for $k\geq 2$ to
\[
\left|\frac{\partial}{\partial t}\intrd \bar{u}\partial_{k}u \psi_jdx\right|\leq \frac{C}{\sqrt{t}}\int_{I_j}|\nabla u|^2+|u|^2dx.
\]
Now, we remark that 
\[
\int_{I_j}\left(|\nabla u|^2+|u|^2\right)dx\leq \int_{I_j}|\nabla R|^2+|R|^2dx+\nhud{u-R}.
\]
Recall that by hypothesis we have 
\[
\nhu{u-R}=\nhu{v}\leq e^{-\alpha\w_\star^{\frac{1}{2}} v_\star t}.
\]
In addition, the decay properties of each $\Phi_k$ and the definition of $I_j$ imply 
\[
\int_{I_j}\left(|\nabla u|^2+|u|^2\right)dx\leq Ce^{-2\alpha\w_\star^{\frac{1}{2}} v_\star t}.
\]
Consequently, 
\[
\left|\frac{\partial}{\partial t}\intrd|u(t)|^2\psi_j(t,x)dx\right|+\left|\frac{\partial}{\partial t}\intrd \bar{u}\nabla u \psi_jdx\right|\leq \frac{C}{\sqrt{t}}e^{-2 \alpha\w_\star^{\frac{1}{2}} v_\star t}.
\]
Note that the previous inequality is trivial for $j=1$ since $\psi_1=1$ and the mass and momentum are conserved.
Plugging the previous into the expressions of $M_j$ and $P_j$ gives
\[
\left|\frac{\partial}{\partial t}\left(M_j(t,u)+P_j(t,u)\right)\right|\leq \frac{C}{\sqrt{t}}e^{-2\alpha \w_\star^{\frac{1}{2}} v_\star t}
\]
and the desired conclusion readily follows.
\end{proof}

\begin{proof}[Proof of Proposition~\ref{prop:estimate2}]
Let $K=K(\Phi_1,...,\Phi_N)$ and $m:=\sum_{j=1}^N\nu_j$ be given by Lemma~\ref{lem:coercivity-global}. Since $\nld{ X_j^k(t)}=1$ for any $t,j,k$, by Lemma~\ref{lem:step1}, there exists $v_\sharp=v_\sharp(\Phi_1,...,\Phi_N)$ such that if $v_\star>\alpha^{-1}v_\sharp$ we have for $j=1,...,N$, $k=1,...,\nu_j$ that
\begin{equation}\label{eq:prop:estimate2-1}
\psld{v(t)}{ X^k_j(t)}^2\leq \nldd{v(t)}\leq \frac{1}{2mK}e^{-2\alpha\w_\star^{\frac{1}{2}}  v_\star t}.
\end{equation}
Using Lemma~\ref{lem:estimate-dS} we obtain
\begin{equation}\label{eq:prop:estimate2-2}
\mathcal{S}(t,u(t))-\mathcal{S}(T_n,u(T_n))\leq \int^{T_n}_t\left| \frac{\partial \mathcal{S}(s,u(s))}{\partial s}\right|ds\leq \frac{C}{\sqrt{t}}e^{-2\alpha\w_\star^{\frac{1}{2}}  v_\star t}.
\end{equation}
Note that since $u_n(T_n)=R(T_n)$ we have
\begin{equation}\label{eq:S(u(T_n)=S(R)}
\mathcal{S}(T_n,u(T_n))-\mathcal{S}(T_n,R(T_n))=0
\end{equation}
By Lemma~\ref{lem:S-eexpansion}, \eqref{eq:prop:estimate2-2}-\eqref{eq:S(u(T_n)=S(R)} imply
\begin{equation}\label{eq:prop:estimate2-3}
\mathcal{H}(t,v)\leq \frac{Ce^{-2\alpha\w_\star^{\frac{1}{2}}  v_\star t}}{\sqrt{t}}+o(\nhud{v}).
\end{equation}
Combining \eqref{eq:prop:estimate2-1}-\eqref{eq:prop:estimate2-3} and Lemma~\ref{lem:coercivity-global} we get
\[
\nhud{v}\leq \left(\frac{C}{\sqrt{t}}+\frac{1}{2}\right)e^{-2\alpha\w_\star^{\frac{1}{2}}  v_\star t}+o(\nhud{v})
\]
and we easily obtain the desired conclusion if $T_0$ is chosen large enough.
\end{proof}

\section{Non-uniqueness and instability}

In this section, we assume $g\in \mathcal C^\infty$ and (A1)-(A4) are satisfied. We take $N\in\N\setminus\{0,1\}$, and for $j=1,...,N$, $\w_j>0$,
$\g_j\in\R$, $v_j\in\Rd$, $x_j\in\Rd$ and
$\Phi_j\in\hu$ a solution of \eqref{eq:snls} (with $\w_0$ replaced by $\w_j$). Recall that 
\begin{gather*}
R_j(t,x)=\Phi_j(x-v_jt-x_j)e^{i(\frac{1}{2}v_j\cdot
  x-\frac{1}{4}|v_j|^2t+\w_jt+\g_j)},\\
\w_\star=\frac{1}{2}\min\left\{ \w_j,j=1,...,N \right\},\quad
v_\star=\frac{1}{9}\min\left\{|v_{j}-v_k|;j,k=1,...,N,j\neq k
\right\}.
\end{gather*}

\subsection{Construction of approximation profiles}\label{sec:profile}

Since \eqref{eq:nls} is Galilean invariant, we can assume without loss of generality that $v_1=0,\gamma_1=0,x_1=0$.
For notational brevity we drop in this subsection the subscript $1$ indicating that we work we the first excited state. Hence we will write (in this subsection only) $R_1(t,x)=R(t,x)$, $\Phi_1=\Phi$, etc.

Note first $df(z).w = g(|z|^2)w + 2 \Re(z \bar w) g'(|z|^2) z$ is not $\m C$-linear. 
This is why we shall identify $\m C$ with $\m R^2$ and use the notation $a+ib = \begin{pmatrix} a \\ b \end{pmatrix}$ ($a,b \in \m R$), so as to consider operators with real entries.  Given a vector $v \in \m C^2$, we denote $v^+$ and $v^-$ its components (so that if $v$ represents a complex number, $v^+$ is the real part and $v^-$ the imaginary part). To avoid confusion, we will denote with an index whether we consider the operator with $\m C$, $\m R^2$, or $\m \C^2$-valued functions. 

Thus, as we consider
\[ 
\q L_{\m C} v = -i\Delta v - idf(R).v, \quad  L_{\m C} v = -i\Delta v + i\w v - idf(\Phi).v,
\]
and the non-linear operators
\begin{align*} 
 \q N_{\m C} (v) &= i f(R+v) - if(R) - idf(R).v, \\ 
\q M_{\m C} (v)& = e^{-i\w t} \q N(e^{i\w t} v) = if(\Phi+v) - if(\Phi) - idf(\Phi).v,
\end{align*}
then for instance
\[ 
L_{\m R^2} 
\begin{pmatrix}
 v^+ \\ v^- 
\end{pmatrix} 
= 
\begin{pmatrix}
J &  \Delta - \omega +I^- \\
-\Delta + \omega - I^+ & -J
\end{pmatrix}
\begin{pmatrix}
v^+ \\
v^-
\end{pmatrix}.
\]
with $\Phi^+$ and $\Phi^-$ the real and imaginary parts of $\Phi$ and
\[ 
J = 2\Phi^+ \Phi^- g'(|\Phi|^2), \quad  I^\pm = g(|\Phi|^2) + 2{\Phi^{\pm}}^2 g'(|\Phi|^2). 
\]
Now $L_{\m R^2}$ is as an (unbounded) $\m R$-linear operator on $H^2(\m R^d, \m R^2) \to L^2(\m R^d, \m R^2)$. So as to have some eigenfunctions, we can complexify, and we are interested in $L_{\m C^2} : H^2(\m R^d, \m C^2) \to L^2(\m R^d, \m C^2)$, which is a $\m C$-linear operator with real entries.

Let $\alpha>0$ be the decay rate given by Proposition~\ref{thm:decayofeigenfunction} for eigenfunctions of $L$ with eigenvalue $\lambda$ (see (A4)). Possibly taking a smaller value of $\alpha$, we can assume  $\alpha\in (0,\sqrt \w)$. For $\m K = \m R, \m R^2, \m C$ or $\m C^2$, denote
\begin{equation}\label{eq:defspaceqH}
\q H(\m K) = \{ v \in H^\infty(\m R^d, \m K) | \ e^{\alpha |x|} |D^a v| \in L^\infty(\Rd)\text{ for any multi-index }a \}.
\end{equation}
We have gathered in the following proposition some properties of $L_{\m C^2}$ that shall be needed for our analysis. 
\begin{prop}[Properties of $L_{\m C^2}$] \mbox{}
\begin{itemize}
\item[(i)] The eigenvalue $\lambda = \rho + i \theta \in \m C$ of $L_{\m C^2}$ can be chosen with maximal real part. We denote $Z(x) = \begin{pmatrix}Z^+(x) \\ Z^-(x) \end{pmatrix} \in H^2(\m R^d, \m C^2)$ an associated eigenfunction.
\item[(ii)] $\Phi \in \q H(\m R^2)$ and $Z \in \q H(\m C^2)$.
\item[(iii)] Let $\mu \notin \Sp(L_{\m R^2})$, and $A \in \q H(\m C^2)$. Then there exists a solution $X \in \q H(\m C^2)$ to
$(L - \mu I)X =A$, and $(L-\mu I)^{-1}$ is a continuous operator on $\q H(\m C^2)$. 
\end{itemize}
\end{prop}

Exponential decay of eigenvalues of $L$ is a fact of independent interest. Hence we have stated the result under general assumptions in the  \ref{ap:decay} (see Proposition~\ref{thm:decayofeigenfunction}). Notice that we treat all possible eigenvalues (in particular without assuming $|\Im\lambda|<\w$, as it is the case for example in \cite{HuLe07}).

\begin{proof}
(i) It is well known that the spectrum of $L_{\m C^2}$ is composed of essential spectrum on $\{iy,y\in\R,|y|\geq\w\}$ and eigenvalues symmetric with respect to the real and imaginary axes (see e.g. \cite{Gr88,HuLe07}). The set of eigenvalues with positive real part is non-empty due to (A4). As $L_{\m C^2}$ is a compact perturbation of $\begin{pmatrix}0&\Delta-\w\\-\Delta+\w&0\end{pmatrix}$ there exists an eigenvalue $\lambda$ with maximal real part.

(ii) Exponential decay of $\Phi$, $\nabla \Phi$ is a well-known fact (see e.g. \cite{Ca03}). Then using the equation satisfied by $\Phi$, one deduces that $\Phi \in \q H(\m R^2)$. The decay and regularity of the eigenfunction $Z$ rely essentially on the decay and regularity of $\Phi$. Therefore, we leave the proof to  \ref{ap:decay}, Proposition~\ref{thm:decayofeigenfunction} and Proposition~\ref{prop:decayreg}.

(iii) Regularity of $X$ follows from a simple bootstrap argument. For the exponential decay, we use the properties of fundamental solutions of Helmoltz equations (see Proposition~\ref{prop:decayreg}).
\end{proof}

To conclude with the notations, we define the decay class $O(\chi(t))$, which we will use for functions decaying exponentially in time.

\begin{defn}
Let $\xi \in \mathcal C^\infty(\m R^+, H^\infty(\Rd))$ and $\chi:\m R^+ \to (0,+\infty)$. Then we denote
\[ \xi(t) = O(\chi(t))  \quad \text{as } t \to +\infty, \]
if, for all $s \ge 0$, there exists $C(s) >0$ such that
\[ \forall t \ge 0, \quad \| \xi(t) \|_{H^s(\Rd)} \le C(s) \chi(t). \]
\end{defn}

Let $Y_1 := \Re(Z) = \begin{pmatrix} \Re(Z^+) \\ \Re(Z^-) \end{pmatrix}$ and $Y_2 := \Im (Z) = \begin{pmatrix} \Im(Z^+) \\ \Im(Z^-) \end{pmatrix}$. Then $Y_1, Y_2 \in \q H(\m R^2)$, and 
\[
\left\{ \begin{array}{rcl}
L_{\m R^2} Y_1 & = & \rho Y_1  -  \theta Y_2, \\
L_{\m R^2} Y_2 & = & \theta Y_1  +  \rho Y_2.
\end{array} \right.
\]
Denote
\begin{equation}\label{def:Y}
Y(t) = e^{-\rho t} (\cos (\theta t) Y_1 + \sin (\theta t) Y_2).
\end{equation}
\begin{lem}
The function $Y$ verifies for all $t\in\m R$ the following equation.
\begin{equation} \label{eq:Ylinflow}
\partial_t Y + L_{\m R^2} Y =0.
\end{equation}
\end{lem}

\begin{proof}
Indeed, we compute
\begin{align*}
\MoveEqLeft 
\partial_t (e^{-\rho t} (\cos (\theta t) Y_1 + \sin (\theta t) Y_2)), \\ 
& = e^{-\rho t} \left( (-\rho \cos(\theta t) - \theta \sin (\theta t)) Y_1 +(-\rho \sin(\theta t) + \theta \cos(\theta t)) Y_2 \right), \\
\MoveEqLeft 
L_{\m R^2} (e^{-\rho t} (\cos (\theta t) Y_1 + \sin (\theta t) Y_2)) ,\\
& = e^{-\rho t} (\cos (\theta t) L Y_1 + \sin (\theta t) L Y_2)) ,\\
& = e^{-\rho t} (\cos (\theta t) (\rho Y_1 - \theta Y_2) + \sin(\theta t) (\theta Y_1 + \rho Y_2) ,\\
& =  e^{-\rho t} ((\rho \cos (\theta t) + \theta  \sin(\theta t)) Y_1 + (\rho \sin (\theta t) - \theta \cos(\theta t)) Y_2.
\end{align*}
So that $(\partial_t + L_{\m R^2})(Y(t)) =0$.
\end{proof}

\begin{prop} \label{prop:profile}
Let $N_0 \in \m N$ and $a \in \m R$. Then there exists a profile $W^{N_0} \in \mathcal C^\infty([0,{+\infty}), \q H(\m R^2))$, such that as $t \to +\infty$,
\[
\partial_t W^{N_0} + L_{\m R^2} W^{N_0} = \q M_{\m R^2} (W^{N_0}) + O(e^{-\rho (N_0+1) t}),
\]
and $W^{N_0}(t) = a Y(t) + O(e^{-2 \rho t})$.
\end{prop}

\begin{rmk}
Notice that $W^{N_0}(t,x)$ is a real valued vector. If we go back and consider $W^{N_0}$ as a function taking values in $\m C$, we then have, by definition of $\q M$, with $U^{N_0}(t) = R(t) + e^{i\omega t} W^{N_0}(t)$,
\[ i \partial_t U^{N_0} + \Delta U^{N_0} + f(U^{N_0}) = O(e^{-\rho (N_0+1) t}). \]
\end{rmk}

For the proof of Proposition~\ref{prop:profile}, we write $W$ for $W^{N_0}$ (for simplicity in notation) and we look for $W$ in the following form
\begin{equation}\label{eq:deffirstprofile} 
W(t,x) = \sum_{k=1}^{N_0} e^{-\rho k t} \left( \sum_{j=0}^{k} A_{j,k} (x) \cos(j \theta t) + B_{j,k}(x) \sin(j \theta t) \right), 
\end{equation}
where $A_{j,k} = \begin{pmatrix}
A_{j,k}^+ \\ A_{j,k}^-
\end{pmatrix}$ and $B_{j,k} = \begin{pmatrix}
B_{j,k}^+ \\ B_{j,k}^-
\end{pmatrix}$ are some functions of $\q H(\m R^2)$ to be determined.

We start by the expansion of $\q M(W)$.

\begin{claim} \label{claim:nlterm} We have
\[ 
\q M_{\m R^2}(W) 
= \sum_{\kappa=2}^{N_0} e^{-\kappa \rho t} \sum_{j=0}^\kappa \left( \tilde A_{j,\kappa}(x) \cos(j \theta t) + \tilde B_{j,\kappa}(x) \sin(j \theta t) \right) + O(e^{-(N_0+1) \rho t}) 
\]
where $\tilde A_{j,\kappa}, \tilde B_{j,\kappa} \in \q H(\m R^2)$ depend on $A_{l,n}$ and $B_{l,n}$ only for $l \le n \le \kappa-1$.
\end{claim}

\begin{proof}
First we use a Taylor expansion. Due to smoothness of $f$ and $\Phi \in \q H(\m R^2)$, and as $\q M_{\m R^2}$ is at least quadratic in $v$, there exists a polynomial $P_{N_0} \in \q H(\m R^2)[X,Y]$ with coefficients in $\q H(\m R^2)$, and valuation at least 2, such that~:
\[ \q M_{\m R^2}(v) = P_{N_0}(v^+,v^-) + O(|v|^{N_0+1}) = \sum_{m =2}^{N_0} \sum_{j=0}^{m} 
\begin{pmatrix}
P_{j,m}(x) v_+^j v_-^{m-j} \\
Q_{j,m} (x)  v_+^j v_-^{m-j}
\end{pmatrix} + O(v^{N_0+1}), 
\]
where $P_{j,m}, \ Q_{j,m} \in \q H(\m R)$. 

Consider now the term $W_+^nW_-^{m-n}$ and use \eqref{eq:deffirstprofile}.
It writes 
\begin{multline*}
\left( \sum_{k=1}^{N_0} e^{-\rho k t} \left( \sum_{l=0}^k A^+_{l,k} \cos (l\theta t) + B^+_{l,k} \sin (l \theta t) \right) \right)^n \\
\times \left( \sum_{k=1}^{N_0} e^{-\rho k t} \left( \sum_{l=0}^k A^-_{l,k} \cos (l\theta t) + B^-_{l,k} \sin (l \theta t) \right) \right)^{m-n}.
\end{multline*}
Now, the multinomial development gives
\begin{multline*} 
\sum_{\begin{subarray}{c} i_1 + \cdots + i_{N_0} =n \\ j_1 + \cdots + j_{N_0} =m-n \end{subarray}} 
\frac{n!}{i_1! \cdots i_{N_0}!}  \frac{(m-n)!}{j_1! \cdots j_{N_0}!} e^{-\rho t \sum_{k=1}^{N_0} k (i_k+j_k) }  \\
\times \prod_{k=1}^{N_0} \left[\left( \sum_{l=0}^{k} \left( A_{l,k}^+ (x) \cos(l \theta t) + B_{l,k}^+(x) \sin(l \theta t) \right) \right)^{i_k}\right.\\
\times\left.\left( \sum_{l=0}^{k} \left( A_{l,k}^- (x) \cos(l \theta t) + B_{l,k}^-(x) \sin(l \theta t) \right) \right)^{j_k} \right].
\end{multline*}

Fix some $(i_k)_k$, $(j_k)_k$ and define the decay rate $\kappa = \sum_{k=1}^{N_0} k (i_k+j_k)$. Then 
\[ \kappa \ge \sum_{k=1}^{N_0} (i_k+j_k) = n+(m-n)=m \ge 2. \]
The product factor is a trigonometric polynomial in $t$, it can be linearized into a sum of $\sin$ and $\cos$ with frequency $\ell \theta$ and $\ell \le \sum_k k (i_k+j_k) = \kappa$.

Of course, as $W \in \q H(\m R^2)$, the higher order terms (i.e. with $\kappa \ge N_0+1$) all fit into $O(e^{-(N_0+1)\rho t})$.

It is now clear that $\tilde A_{j,\kappa}$ and $\tilde B_{j,\kappa}$ are polynomial in $A_{j,k}$, $B_{j,k}$, $P_{n,m}$, and $Q_{n,m}$. It remains to see that the $A_{j,k}$ or $B_{j,k}$ that intervene (i.e $i_k+ j_k >0$) come with $k \le \kappa -1$. Let $a$ be the maximal index such that $i_a + j_a >0$. Recall $i_1 + \cdots +i_{N_0} + j_1 + \cdots +j_{N_0} = m \ge 2$. If $i_a+j_a \ge 2$, we have $2a \le a (i_a+j_a) \le \kappa$ so that (as $\kappa \ge m \ge 2$) $a \le \kappa-1$. If $i_a+j_a =1$, there exist $b \ge 1$, $b \ne a$, such that $i_b+j_b \ge 1$ and 
\[ \kappa = \sum_k k (i_k+j_k) \ge a (i_a +j_a) + b(i_b + j_b) \ge a +1. \] 
Finally the product has the desired properties.
\end{proof}

\begin{proof}[Proof of Proposition~\ref{prop:profile}]
By definition of $W$, we can compute:
\begin{multline*}
(\partial_t W + L_{\m R^2} W) =  \sum_{k=1}^{N_0} e^{-\rho k t} \left( \sum_{j=0}^{k} (L_{\m R^2} A_{j,k} + j \theta B_{j,k} - k \rho A_{j,k}) \cos(j \theta t) \right. 
\\
\left. \vphantom{\sum_N^N} + (L_{\m R^2} B_{j,k} - j \theta A_{j,k} - k \rho B_{j,k}) \sin(j \theta t) \right).
\end{multline*}

From the computations of Claim~\ref{claim:nlterm}, it suffices to solve for all $0\leq j \le k \le N_0$
\begin{equation} \label{eq:Ajk}
\left\{ \begin{array}{r}
L_{\m R^2} A_{j,k} + j\theta B_{j,k} - k \rho A_{j,k} =  \tilde A_{j,k}, \\
L_{\m R^2} B_{j,k} - j \theta A_{j,k} - k \rho B_{j,k} = \tilde B_{j,k}.
\end{array} \right.
\end{equation}
Obviously, one starts to solve for $k=1$, then from this $k=2$ etc. so that at all stages $\tilde A_{j,k}$ and $\tilde B_{j,k}$ are well defined (remark that $\tilde A_{j,1} = \tilde B_{j,1} =0$).

We initialized the induction process by setting $A_{1,1} = a Y_1$, $B_{1,1} = a Y_2$, and $A_{0,1}=B_{0,1}=0$.
Assume that $A_{j,k}$ and $B_{j,k}$ are constructed up to $k \le k_0-1$ and belong to $\q H(\m R^2)$, we now construct $A_{j,k_0}$, $B_{j,k_0}$ for all $j\le k_0$. By Claim~\ref{claim:nlterm}, all $\tilde A_{j,k_0}$ and $\tilde B_{j,k_0}$ are constructed for $j \le k_0$ and belong to $\q H(\m R^2)$.

Consider now the operator $L_{j,k_0} = L_{\m C^2} - (k_0 \rho + i j\theta) \Id$, $L_{j,k_0}: \q H(\m C^2) \to \q H(\m C^2)$. As $e = \rho + i \theta$ is an eigenvalue of $L_{\m C^2}$ with maximal real part, for all $k_0 \ge 2$ and all $j$, $k_0 \rho + i j\theta \notin \Sp(L)$ so that $L_{j,k_0}$ is invertible. Let $X = L_{j,k_0}^{-1} (\tilde A_{j,k_0} + i \tilde B_{j,k_0})$, and define $C := \Re(X) = \begin{pmatrix} \Re(X^+) \\ \Re(X^-) \end{pmatrix}$, $D := \Im(X) = \begin{pmatrix} \Im(X^+) \\ \Im(X^-) \end{pmatrix}$, so that $C,D \in \q H(\m R^2)$ and $X = C +iD$. Then we compute
\begin{align*}
\tilde A_{j,k_0} + i \tilde B_{j,k_0} & = L_{j,k_0} (C+iD) \\
& = L_{\m R^2} C + i L_{\m R^2} D - k_0 \rho C - i k_0 D - i j \theta C + j \theta D  \\
& = (L_{\m R^2}C - k_0 \rho C + j \theta D) + i (L_{\m R^2}D - j \theta C - k_0 \rho D).
\end{align*}
Hence $A_{j,k_0} = C$ and $B_{j,k_0} =D$ are solutions to the system \eqref{eq:Ajk}.
\end{proof}

We now switch back notation from vector valued functions to complex valued functions and summarize what we have obtained. We use again the subscript $1$. Hence we can consider $V_1^{N_0}$, $U_1^{N_0}$ defined by 
\begin{equation*}
V_1^{N_0}(t,x): =  e^{i \omega t} W^{N_0}(t,x), \quad U_1^{N_0}(t,x)  := R_1(t,x) + V_1^{N_0}(t,x).
\end{equation*}
Then we define
\begin{align*}
Err_1^{N_0}(t,x) & :=  i \partial_t U_1^{N_0} + \Delta U_1^{N_0} + f(U_1^{N_0}) \\
& = i \partial_t V_1^{N_0} + \Delta V_1^{N_0} + f(R_1(t) + V_1^{N_0}) - f(R_1(t)) \\
& = i ( \partial_t V_1^{N_0} + \q L_{\m C} V_1^{N_0} - \q N_{\m C}(V_1^{N_0})) \\
& = i e^{i \omega t} (\partial_t W^{M_0} + L_{\m C} W^{N_0} - \q M_{\m C}(W^{N_0})).
\end{align*}
By Proposition~\ref{prop:profile}, $Err_1^{N_0}(t,x) = O(e^{-(N_0+1)\rho t})$. Also, from \eqref{eq:deffirstprofile} we deduce $V_1^{N_0}(t) = a e^{i\omega t} Y(t) + O(e^{-2\rho t})$, so that for all $s \ge 0$, there exists $C(N_0,s)$ such that
\begin{equation} \label{eq:Vdecay}
\forall t \ge 0, \quad \| V_1^{N_0}(t) \|_{H^s(\Rd)} \le C(N_0,s) e^{-\rho t}.
\end{equation}

\subsection{Proofs of Theorems~\ref{thm:2} and~\ref{thm:3}}

\begin{proof}[Proof of Theorem~\ref{thm:2}]
Let $N_0$ to be determined later, we do a fixed point around $U_1^{N_0}(t)$. Suppose $u = U_1^{N_0}(t) + w(t)$ (with $w(t) \to 0$ as $t \to +\infty$) is a solution to \eqref{eq:nls}, then 
\[ i\partial_t w + \Delta w + f(U_1^{N_0}+w) - f(U_1^{N_0}) - Err_1^{N_0}(t) =0 \]
From this, Duhamel's Formula gives, for $t \le s$,
\[ 
w(s) = e^{i\Delta (s-t)} w(t) + i\int_t^s \! e^{i \Delta(s-\tau)} \left(f((U_1^{N_0}+w)(\tau)) - f(U_1^{N_0}(\tau)) -  Err_1^{N_0}(\tau)\right) d\tau, 
\]
so that
\[ 
e^{-i\Delta s} w(s) = e^{-i\Delta t} w(t) + i\int_t^s e^{-i\Delta \tau} \left(f((U_1^{N_0}+w)(\tau)) - f(U_1^{N_0}(\tau)) -  Err_1^{N_0}(\tau)\right) d\tau.
\]
Letting $s \to +\infty$, as $w(s) \to 0$, we are looking for a solution to the fixed point equation
\[  
w(t) = - i\int_t^{+\infty} e^{i\Delta (t-\tau)} (f((U_1^{N_0}+w)(\tau)) - f(U_1^{N_0}(\tau)) -  Err_1^{N_0}(\tau)) d\tau.
\]
Hence, we define the map
\[ 
v \mapsto \Psi(v) = - i\int_t^{+\infty} e^{i\Delta (t-\tau)} (f((R_1+V_1^{N_0}+v)(\tau)) - f((R_1+V_1^{N_0})(\tau)) -  Err_1^{N_0}(\tau)) d\tau. 
\]
Fix $\sigma > \frac{d}{2}$, so that $H^\sigma(\Rd)$ is an algebra, and let $B,T_0$ to be determined later.  
For $w\in\mathcal{C}((T_0,+\infty),H^\sigma(\Rd))$ define 
\[
 \| w \|_{ X_{T_0,N_0}^\sigma} = \sup_{t \ge T_0} e^{(N_0+1) \rho t} \| w(t) \|_{H^\sigma(\Rd)},
\]
to be the norm of the Banach space
\[ 
X_{T_0,N_0}^\sigma:=\left\{ w\in\mathcal{C}((T_0,+\infty),H^\sigma(\Rd))\middle|  \| w \|_{ X_{T_0,N_0}^\sigma}<+\infty \right\}.
\]
Consider the ball of radius $B$ of $X_{T_0,N_0}^\sigma$
\[ 
X_{T_0,N_0}^\sigma(B) := \left\{ w \in X_{T_0,N_0}^\sigma \middle| \| w \|_{ X_{T_0,N_0}^\sigma}\leq B \right\}. 
\]
By \eqref{eq:Vdecay}, we can assume $T_0$ is large enough so that
\[ 
\|V^{N_0}_1\|_{H^\sigma(\Rd)} \le 1 \text{ and also }B e^{-(N_0+1) \rho T_0} \le 1. 
\]
Our problem is to find a fixed point for $\Psi$, we will find it in $X_{T_0,N_0}^\sigma(B)$ for adequate parameters.

Notice that for $t \ge T_0$, $\| V_1^{N_0}(t) \|_{H^\sigma(\Rd)} \le 1$. Hence, we will always work in the $H^\sigma(\Rd)$-ball of radius $r_\sigma = \| \Phi_1 \|_{H^\sigma(\Rd)} + 2$. Due to $\mathcal{C}^{\sigma+1}$ smoothness of $f$, there exists a constant $K_\sigma$ such that
\[
\forall a,b \in B_{H^\sigma(\Rd)}(r_\sigma), \quad \|f(a) - f(b) \|_{H^\sigma(\Rd)} \le K_\sigma \|a-b\|_{H^\sigma(\Rd)}.
\]
In particular, for all $t$,
\[ \| f(R_1(t) + V_1^{N_0}(t) + v) - f(R_1(t)+ V_1^{N_0}(t)) \|_{H^\sigma(\Rd)} \le K_\sigma \| v \|_{H^\sigma(\Rd)}. \]
Hence, as $e^{i \Delta(t-s)}$ is an isometry in $H^\sigma(\Rd)$, for any $v \in X_{T_0,N_0}^\sigma(B)$ we have
\begin{align*}
\MoveEqLeft \| \Psi(v)(t) \|_{H^\sigma(\Rd)} \\
& = \left\| \int_t^{+\infty} e^{i \Delta(t-\tau)} \left[f(R_1 + V_1^{N_0} + v) - f(R_1 + V_1^{N_0}) - Err_1^{N_0} \right](\tau) d\tau \right\|_{H^\sigma(\Rd)} \\
& \le \int_t^{+\infty} ( \| f(R_1 + V_1^{N_0} + v) - f(R_1 + V_1^{N_0})  \|_{H^\sigma(\Rd)} + \| Err_1^{N_0}(\tau) \|_{H^\sigma(\Rd)} ) d\tau \\
& \le \int_t^{+\infty} (K_\sigma  \| v \|_{H^\sigma(\Rd)} + C(N_0,\sigma) e^{-(N_0+1) \rho \tau} ) d\tau \\
& \le \frac{K_\sigma B + C(N_0,\sigma)}{(N_0+1) \rho} e^{-(N_0+1) \rho t}.
\end{align*}
First choose $N_0$ large enough so that $\frac{K_\sigma}{(N_0+1) \rho} \le \frac{1}{2}$. Then choose $B = 2 \frac{C(N_0,\sigma)}{(N_0+1)\rho}$. Finally choose $T_0$ large enough so that $\quad C(N_0,\sigma) e^{-\rho T_0} \le 1.$
Hence we get
\[ \| \Psi(v)(t) \|_{H^\sigma(\Rd)} \le B e^{-(N_0+1) \rho t}. \]
 This shows that $\Psi$ maps $X^\sigma_{T_0,N_0}(B)$ to itself. 
Let us now show that $\Psi$ is a contraction in $X^\sigma_{T_0,N_0}(B)$. Let $v,w \in X^\sigma_{T_0,N_0}(B)$ then we have
\[  \Psi(v)(t) - \Psi(w)(t)   = -i\int_t^{+\infty} e^{i\Delta (t-s)} (f(R_1 + V_1^{N_0} + v) - f(R_1 +V_1^{N_0} + w)) ds. \]
As previously, we have
\begin{align*}
\MoveEqLeft 
e^{(N_0+1) \rho t} \| \Psi(v)(t) - \Psi(w)(t) \|_{H^\sigma(\Rd)}\\
& = e^{(N_0+1) \rho t} \left\| \int_t^{+\infty} e^{i\Delta (t-s)} (f(R_1 + V_1^{N_0} + v) - f(R_1 +V_1^{N_0} + w)) ds \right\|_{H^\sigma(\Rd)} \\
& \le e^{(N_0+1) \rho t} \int_t^{+\infty} \| f(R_1 + V_1^{N_0} + v) - f(R_1 +V_1^{N_0} + w) \|_{H^\sigma(\Rd)} ds \\
& \le e^{(N_0+1) \rho t} \int_t^{+\infty} K_\sigma \| v(s) - w(s) \|_{H^\sigma(\Rd)} ds \\
& \le K_\sigma e^{(N_0+1) \rho t} \int_t^{+\infty} e^{-(N_0+1)\rho s} \| v - w \|_{X^\sigma_{T_0,N_0}} ds \\
& \le K_\sigma e^{(N_0+1) \rho t} \| v - w \|_{X^\sigma_{T_0,N_0}} \frac{e^{-(N_0+1) \rho t}}{(N_0+1) \rho} \\
& \le \frac{K_\sigma}{(N_0+1) \rho}  \| v - w \|_{X^\sigma_{T_0,N_0}}.
\end{align*}
Taking the supremum over $t \ge T_0$, we deduce that
\[ \| \Psi(v) - \Psi(w) \|_{X^\sigma_{T_0,N_0}} \le \frac{K_\sigma}{(N_0+1) \rho} \| v - w \|_{X^\sigma_{T_0,N_0}} \le \frac{1}{2} \| v - w \|_{X^\sigma_{T_0,N_0}} . \]
Hence, $\Psi$ is a contraction on $X^\sigma_{T_0,N_0}(B)$, and has a unique fixed point $\bar v$. Notice that we have obtained a unique fixed point for any $\sigma \ge \frac{d}{2}$: from this we deduce that $\bar v$ does not depend on $\sigma$, and hence, $\bar v \in \mathcal C^\infty([T_0,{+\infty}), H^\infty(\Rd))$. Then $\bar u = R_1 + V_1^{N_0} + \bar v$ is the desired solution.
\end{proof}

\begin{proof}[Proof of Theorem~\ref{thm:3}]
The proof is essentially a generalization of that of Theorem~\ref{thm:2}. Let $v_\natural$ to be fixed later and assume that $v_\star>v_\natural$. Let $N_0$ to determined later and $a \in \m R$, from this we dispose of a profile $V^{N_0}_1(t)$, $U_1^{N_0}(t)$, an error term $Err_1^{N_0}(t)$ associated to $R_1(t)$, and an eigenvalue 
 $\lambda = \rho + i \theta$ of $L$. 
We look for a solution of the form $u(t) = U^{N_0}_1(t) + \sum_{j \ge 2} R_j(t) + w(t)$. Then $w$ satisfies
\[ 
i \partial_t w + \Delta w + f(U^{N_0}_1 + \sum_{j \ge 2} R_j + w) - f(U_1^{N_0}) - \sum_{j \ge 2} f(R_j) - Err^{N_0}_1 =0. 
\]
Hence considering the map
\[
v \mapsto \Psi(v) 
= - i\int_t^{+\infty} e^{i\Delta (t-s)} (f((U^{N_0}_1 + \sum_{j \ge 2} R_j+v)(s)) - f(U_1^{N_0}(s))
 - \sum_{j \ge 2} f(R_j(s)) -  Err^{N_0}_1(s)) ds, 
\]
we are looking for a fixed point for $\Psi$, in the set $X_{T_0,N_0}^\sigma(B)$ (defined in the proof of Theorem~\ref{thm:2})
for adequate parameters $T_0,N_0,B, \sigma$. Let $\sigma > \frac{d}{2}$. As previously, let $T_0$ large enough so that $\| V_1^{N_0} (t) \|_{H^s(\Rd)} \le 1$ for $t \ge T_0$, and $B e^{-(N_0+1) \rho T_0} \le 1$, so that we remain in a ball of radius 1 in $H^\sigma(\Rd)$.

Using exponential localization of the solitons $R_j$ and of the profile $U^{N_0}_1$, we deduce as in the proof of Theorem~\ref{thm:2} that for some $K_\sigma = K(f, \| U^{N_0}_1\|_{H^\sigma(\Rd)} + \sum_{j \ge 2} \| R_j \|_{H^\sigma(\Rd)} +1)$, we have
\[
\| f((U^{N_0}_1 + \sum_{j \ge 2} R_j+v)) - f(U_1^{N_0}) - \sum_{j \ge 2} f(R_j) \|_{H^{\sigma}(\Rd)} 
\le K_\sigma \| v \|_{H^\sigma(\Rd)} + O(e^{- 2 \alpha\sqrt{\w_\star} v_\star t}), 
\]
possibly by taking a smaller value of $\w_\star$ such that $\w_\star \le \alpha_1$, where $\alpha_1$ is the (exponential) decay rate of $U^{N_0}_1$. Notice that $\alpha_1$ is independent of $N_0$, due to the construction of $U^{N_0}_1$. Hence we have as in Theorem~\ref{thm:2}:
\begin{align*}
\MoveEqLeft 
\| \Psi(v)(t) \|_{H^\sigma(\Rd)} \\
& = \left\| \int_t^{+\infty} e^{i \Delta(t-s)} (f(U^{N_0}_1 + \sum_{j \ge 2} R_j+v) - f(U_1^{N_0}) - \sum_{j \ge 2} f(R_j) -  Err^{N_0}_1) ds, \right\|_{H^\sigma(\Rd)} \\
& \le \int_t^{+\infty}  \| f(U^{N_0}_1 + \sum_{j \ge 2} R_j+v) - f(U_1^{N_0}) - \sum_{j \ge 2} f(R_j)  \|_{H^\sigma(\Rd)} + \| Err^{N_0}_1 \|_{H^\sigma(\Rd)} ds \\
& \le \int_t^{+\infty} (K_\sigma  \| v \|_{H^\sigma(\Rd)} + C(N_0,\sigma) e^{-(N_0+1) \rho s} + C(\sigma) e^{- 2\alpha \sqrt{\w_\star} v_\star s} ) ds \\
& \le \frac{K_\sigma B + C(N_0,\sigma)}{(N_0+1) \rho} e^{-(N_0+1) \rho t} + \frac{C(\sigma)}{2\alpha \sqrt{\w_\star} v_\star} e^{- 2\alpha \sqrt{\w_\star} v_\star t}.
\end{align*}
First choose $N_0$ large enough so that $\frac{K_\sigma}{(N_0+1) \rho} \le \frac{1}{3}$ and set $B:=\frac{3C(N_0,\sigma)}{(N_0+1)\rho}$. Recall that $v_\star>v_\natural$. We chose $v_\natural$ large enough so that from the choice of $\w_\star, v_\star$, we have 
\[ \frac{C(\sigma)}{2\alpha \sqrt{\w_\star} v_\star} \le \frac{B}{3}, \text{ and } 2\alpha \sqrt{\w_\star} v_\star \ge (N_0+1) \rho. \]
Finally choose $T_0$ large enough so that
\[ B e^{-(N_0+1) \rho T_0} \le 1,\quad \text{and} \quad C(N_0,\sigma) e^{-\rho T_0} \le 1. \]
From this, $\| \Psi(v)(t) \|_{H^\sigma(\Rd)} \le B e^{-(N_0+1) \rho t}$ for $t \ge T_0$, i.e. $\Psi$ maps $X^\sigma_{T_0,N_0}(B)$ to itself. Similar computations show that $\Psi$ is a contracting map, so that it has a unique fixed point $\bar w$. Again as in Theorem~\ref{thm:2}, $\bar w$ does not depend on $\sigma$ and $\bar w \in \mathcal C^\infty([T_0,{+\infty}),H^\infty(\Rd))$. Then $\bar u = U^{N_0}_1 + \sum_{j \ge 2} R_j(t) + \bar w(t)$ fulfills the requirements.
\end{proof}

\appendix

\section{Exponential decay of eigenfunctions to matrix Schrödinger operators}\label{ap:decay}

We consider an operator $L:H^2(\Rd,\C^2)\subset L^2(\Rd,\C^2)\to L^2(\Rd,\C^2)$ of the form
\[ 
L
= 
\begin{pmatrix}
W_1 & - \Delta + \omega + V_1 \\
\Delta - \omega + V_2 & W_2
\end{pmatrix}
\]
where $\w>0$ and $V_1,V_2,W_1,W_2$ are complex-valued potentials satisfying the following assumptions.
\begin{itemize}
\item[(VW1)] There exists $q\in(\max\{2,\frac{d}{2}\},+\infty]$ such that $V_k,W_k\in L^q(\Rd)$ for $k=1,2$.
\item[(VW2)] $\lim_{|x|\to+\infty}V_k(x)=\lim_{|x|\to+\infty}W_k(x)=0$ for $k=1,2$.
\end{itemize}
Assumptions (VW1)-(VW2) are probably not optimal, but they are sufficient in the context in which we want to apply the following Proposition~\ref{thm:decayofeigenfunction}. 

Our goal is to prove that if $L$ has an eigenvalue which does not belong to the set $\{iy,y\in\R, |y|\geq\w\}$ (which is the essential spectrum of $L$, see e.g. \cite{HuLe07}) then the corresponding eigenvectors are exponentially decaying at infinity. Note that it was previously known only for eigenfunctions corresponding to eigenvalues lying in the  strip $\{z\in\C,|\Im(z)|<\w\}$ and with a restricted class of potentials (see \cite{HuLe07}).

\begin{prop}\label{thm:decayofeigenfunction}
Assume that (VW1)-(VW2) hold. Take  
$u,v\in H^2(\Rd,\C)$, $\lambda\in\C\setminus\{iy,y\in\R, |y|\geq\w\}$, and suppose that for 
$
U:=\begin{pmatrix}
   u\\v
  \end{pmatrix}
$
we have $LU=\lambda U$. Then there exist $C>0$ and $\alpha>0$ such that for all $x\in\Rd$ we have
\[
|u(x)|+|v(x)|\leq Ce^{-\alpha|x|}.
\]
\end{prop}

Our proof consists in obtaining estimates on fundamental solutions to Helmholtz equations and considering the eigenvalue problem $LU=\lambda U$ as an inhomogeneous problem.

\subsection{Fundamental solutions}

For a given $\mu\in\C$, a fundamental solution of the Helmholtz equation in $\Rd$ is a solution of 
\[
(-\Delta-\mu)g^d_\mu=\delta_0.
\]
Setting $\nu:=\frac{d-2}{2}$ fundamental solutions of the Helmholtz equation are given by
\[
g^d_\mu(x):=\frac{i\pi\mu^{\frac{\nu}{2}}}{2|x|^\nu(2\pi)^{\frac{d}{2}}}H^1_\nu(\sqrt{\mu}|x|),
\]
where $H^1_\nu$ is the first Hankel function (see e.g. \cite{AbSt64}).
For $\mu=\rho e^{i\theta}$ with $\rho\geq0$ and $\theta\in[0,2\pi)$ we defined $\sqrt{\mu}$  by $\sqrt{\mu}:=\rho^{\frac{1}{2}}e^{i\frac{\theta}{2}}$. Defining $\sqrt{\cdot}$ in this way ensures in particular that $g^d_\mu$ is square integrable for $\mu\not\in\R^+$. The fundamental solutions $g^d_\mu$ verify the recurrence relation 
\[
g^{d+2}_\mu(x)=-\frac{\frac{\partial}{\partial r}g_\mu^d(x)}{2\pi |x|}.
\]
We deduce the following formula for the fundamental solution. For $d=j+2l$ where $j=1,2$ and $l\in\N\setminus\{0\}$, we have
\begin{equation}\label{eq:formulaalld}
g_\mu^{j+2l}=\sum_{k=1}^l a_l^k(-1)^k(g^j_\mu)^{(k)}|x|^{-2l+k},
\end{equation}
where the coefficients $(a_l^k)$ are positive and the exponent $(k)$ denotes the $k$\textsuperscript{th} derivative.

\begin{lem}[Estimates on fundamental solutions]\label{lem:estimateonfundamentalsol}
Let $\mu\in\C\setminus\R^+$. Then there exists $\tau>0$ and $C>0$ such that 
\[
|g^d_\mu(x)|\leq Cg^d_{-\tau}(x)\mbox{ for all }x\in\Rd\setminus\{0\}.
\]
In particular, $g^d_\mu$ is exponentially decaying at infinity with decay rate $\sqrt{\tau}$, i.e. $|g^d_\mu(x)|\leq Ce^{-\sqrt{\tau}|x|} $ for $|x|$ large enough.  
\end{lem}

We separated the proof of Lemma~\ref{lem:estimateonfundamentalsol} into two proofs depending on the oddness of $d$.

\begin{proof}[Proof for odd $d$]
We have $\sqrt{\mu}=\rho^{\frac{1}{2}}e^{i\frac{\theta}{2}}$. Choose $\tau>0$ such that $\sqrt{\tau}=\rho^{\frac{1}{2}}\sin\frac{\theta}{2}$. It is well-known that $g^1_\mu(x)=\frac{i}{2\sqrt{\mu}}e^{i\sqrt{\mu}|x|}$. It follows from easy computations that
\[
 \left|g^1_\mu(x)\right|\leq \frac{1}{2\sqrt{\rho}}e^{-\rho^{\frac{1}{2}}\sin{\frac{\theta}{2}}|x|}.
\]
Since 
\[
g^1_{-\tau}(x)=\frac{1}{2 \sqrt{\rho}\sin{\frac{\theta}{2}}}e^{-\rho^{\frac{1}{2}}\sin{\frac{\theta}{2}}|x|}
\]
this readily implies that for all $x\in\Rd$ we have
\[
|g^1_\mu(x)|\leq Cg^1_{-\tau}(x),
\]
which proves the lemma for $d=1$. 

Similar calculations lead to
\begin{equation}\label{eq:estimatederivativefundsolodd}
|(g^1_\mu)^{(k)}|\leq C(-1)^k(g^1_{-\tau})^{(k)}\text{ for all }k\in\N. 
\end{equation}
Assume now that $d\geq 3$ and take $l\in\N\setminus\{0\}$ such that $d=1+2l$.
Combining \eqref{eq:formulaalld} and \eqref{eq:estimatederivativefundsolodd}  gives
\[
|g_{\mu}^{1+2l}(x)|\leq Cg_{-\tau}^{1+2l}(x)\mbox{ for all }x\in\Rd\setminus\{0\},
\]
which is the desired conclusion.
\end{proof}

\begin{proof}[Proof for even $d$]
Let $\nu\in\N$ and $z\in\C$. We have the following asymptotic expansions on the Hankel functions (see \cite{AbSt64}).
\begin{alignat*}{2}
iH_0^1(z)&\approx  -\frac{2}{\pi}\ln(z)&\quad&\mbox{ for }|z|\mbox{ close to }0,\\
iH_\nu^1(z)&\approx \frac{\nu!z^{-\nu}}{2^{-\nu}\pi}&&\mbox{ for }|z|\mbox{ close to }0,\nu\neq0, \\
H_\nu^1(z)&\approx \sqrt{\frac{2}{\pi z}}e^{i(z-\frac{\nu\pi}{2}-\frac{\pi}{4})}&&\mbox{ for }|z|\mbox{ close to }+\infty.
\end{alignat*}
Therefore, we can infer the following estimates on the fundamental solutions. Recall that $d=2+2\nu$ and $\mu=\rho e^{i\theta}$. 
\begin{alignat}{2}
|g^2_\mu(x)|&\leq C|\ln(\rho^{\frac{1}{2}}|x|)|&\quad&\mbox{ for }|x|\mbox{ close to }0,\label{eq:estimatefundsol1}\\
|g^d_\mu(x)|&\leq C|x|^{-\nu}&&\mbox{ for }|x|\mbox{ close to }0,\nu\neq0, \\
|g^d_\mu(x)|&\leq C|x|^{-(\nu+1)}e^{-\rho^{\frac{1}{2}}\sin(\frac{\theta}{2})|x|}&&\mbox{ for }|x|\mbox{ close to }+\infty.
\end{alignat}
For $\tau>0$, the function $g^d_{-\tau}$ verifies $g^d_{-\tau}>0$ and
\begin{alignat}{2}
g^2_{-\tau}(x)&\approx  C|\ln(\tau^{\frac{1}{2}}|x|)|&\quad&\mbox{ for }|x|\mbox{ close to }0,\\
g^d_{-\tau}(x)&\approx C|x|^{-\nu}&&\mbox{ for }|x|\mbox{ close to }0,\nu\neq0 \\
g^d_{-\tau}(x)&\approx C|x|^{-(\nu+1)}e^{-\tau^{\frac{1}{2}}|x|}&&\mbox{ for }|x|\mbox{ close to }+\infty\label{eq:estimatefundsol6}.
\end{alignat}
Choose $\tau>0$ such that $\tau^{\frac{1}{2}}=\sqrt{\rho}\sin{\frac{\theta}{2}}$. Then we infer from \eqref{eq:estimatefundsol1}-\eqref{eq:estimatefundsol6} and the continuity of fundamental solutions that there exists $C>0$ such that
\[
|g^d_\mu(x)|\leq Cg^d_{-\tau}(x)\mbox{ for all }x\in\Rd\setminus\{0\},
\]
which finishes the proof.
\end{proof}

\subsection{Exponential decay}

We start with a regularity result on eigenfunctions.

\begin{lem}\label{lem:regularityeigenfunctions}
Assume that (VW1) is satisfied.
Take $\lambda\in\C\setminus\{iy,y\in\R, |y|\geq\w\}$, 
$u,v\in H^2(\Rd,\C)$ and assume that for 
$
U:=\begin{pmatrix}
   u\\v
  \end{pmatrix}
$
we have $LU=\lambda U$. Then $u,v\in W^{2,r}(\Rd)$ for any $r\in[2,q]$. In particular, $u,v\in\mathcal{C}^0(\Rd)$ and $\lim_{|x|\to +\infty}u(x)=\lim_{|x|\to +\infty}v(x)=0$.
\end{lem}

\begin{proof}
The result follows from a classical bootstrap argument. 
Let the sequence $(r_n)$ be defined by
\[
 \left\{
\begin{array}{rcl}
r_0&=&2,\\
\frac{1}{r_{j+1}}&=&\frac{1}{q}+\frac{d-2r_j}{dr_j},
\end{array}
\right.
\]
where $q$ is given by (VW1). An elementary analysis of $(r_j)$ shows that there exists $j_0$ such that for all $0\leq j< j_0$ we have $r_{j+1}>r_j$, $\frac{d-2r_j}{dr_j}>0$ and $\frac{d-2r_{j_0}}{dr_{j_0}}<0$.

By induction, it is easy to see that for all $j=0,...,j_0$ we have $u,v\in W^{2,r_{j}}(\Rd)$. For $j=0$ it is by definition of $u,v$. Take any $0\leq j<j_0$ and assume that $u,v\in W^{2,r_{j}}(\Rd)$ . 
Since $\frac{d-2r_j}{dr_j}>0$, by Sobolev embeddings we infer that $u,v\in L^{\frac{dr_j}{d-2r_j}}(\Rd)$. Then, (VW1) and Hölder inequality imply
\[
W_1u,V_1v,V_2u,W_1v\in L^{r_{j+1}}(\Rd).
\]
Combined with $U=(u,v)^T$ satisfying $LU=\lambda U$, this leads to $u,v\in W^{2,r_{j+1}}(\Rd)$. 

In particular, we have $u,v\in W^{2,r_{j_0}}(\Rd)$. Since $\frac{d-2r_{j_0}}{dr_{j_0}}<0$, from Sobolev embeddings we infer $u,v\in L^{\infty}(\Rd)$. Then by (VW1) we get
\[
W_1u,V_1v,V_2u,W_1v\in L^{q}(\Rd).
\]
As before, combined with $LU=\lambda U$, this leads to $u,v\in W^{2,q}(\Rd)$. The conclusion follows by interpolation.
\end{proof}

For the rest of the proof, it is easier to work with the operator
\[
L':=iPLP^{-1}=
\begin{pmatrix}
-\Delta+\w+ V'_1 &  W'_1\\
W'_2 &\Delta-\w+V'_2
\end{pmatrix},
\]
where $P=\begin{pmatrix}1&i\\1&-i\end{pmatrix}$. The potentials $V'_1,V'_2,W'_1,W'_2$ verify also (VW1)-VW2). The spectrum of $L'$ is $\Sp(L')=\Sp(iPLP^{-1})=i\Sp(L)$. Hence if $\lambda\in\C$ is an eigenvalue of $L$ with eigenvector $U$ then $\lambda':=i\lambda$ is an eigenvalue of $L'$ with eigenvector $U'=\begin{pmatrix}u'\\ v' \end{pmatrix}:=PU$.

Write $L'-\lambda' I=H+K$ where
\[
H:=
\begin{pmatrix}
-\Delta+\w-\lambda' & 0\\
0 &\Delta-\w-\lambda'
\end{pmatrix}\mbox{ and }
K:=
\begin{pmatrix}
V'_1 & W'_1\\
W'_2 &V'_2
\end{pmatrix}.
\]
Take
\[
F:=\begin{pmatrix}
    f_1\\f_2
   \end{pmatrix}
:=KU'=\begin{pmatrix}
      V'_1u+W'_1v\\W'_2u+V'_2v
     \end{pmatrix}.
\]
It is well known that we can represent $u'$ and $v'$ in the following way
\[
u'=g^d_{-\w+\lambda'}*f_1\mbox{ and }v'=-g^d_{-\w-\lambda'}*f_2.
\]

Let $\mu_1:=-\w+\lambda'$ and $\mu_2:=-\w-\lambda'$. From the assumptions on $\lambda'$ we infer that $\mu_1,\mu_2$ satisfy the hypothesis of Lemma~\ref{lem:estimateonfundamentalsol}. Let $\tau_1$, $\tau_2$ be given by Lemma~\ref{lem:estimateonfundamentalsol} and set $\tau:=\min\{\tau_1,\tau_2\}$.
Take 
\begin{gather*}
\tilde{F}:=\begin{pmatrix}
    \tilde f_1\\ \tilde f_2
   \end{pmatrix}:=
\begin{pmatrix}
    |f_1|\\ |f_2|
   \end{pmatrix},
\\
\tilde{u}:=g^d_{-\tau}*\tilde f_1\mbox{ and } \tilde{v}:=g^d_{-\tau}*\tilde f_2.
\end{gather*}

\begin{claim}\label{cl:2}
There exists $C>0$ such that
\[
|u'|\leq C \tilde u\mbox{ and }|v'|\leq C\tilde v.
\]
\end{claim}

\begin{proof}
This readily follows from Lemma~\ref{lem:estimateonfundamentalsol}.
\end{proof}

\begin{lem}\label{lem:decay}
Set $w:=\tilde u+\tilde v$. There exists $C>0$ and $\alpha>0$ such that 
\[
w(x)\leq Ce^{-\alpha|x|}\mbox{ for all }x\in\Rd.
\]
\end{lem}

The proof of Lemma~\ref{lem:decay} follows closely the proof of Theorem 1.1 in \cite{DeLi07}. 

\begin{proof}
Set $f:=\tilde f_1+\tilde f_2$. We first note that $w\in \mathcal{C}^0(\Rd)$. Indeed, by definition $w$ satisfies
\begin{equation}\label{eq:Helmholtzforw}
-\Delta w+\tau w=f.
\end{equation}
Since, by (VW1) and Lemma~\ref{lem:regularityeigenfunctions}, $f\in L^q(\Rd)$ this implies $w\in W^{2,q}(\Rd)$ and in particular $w\in\mathcal{C}^0(\Rd)$.

Now, we claim that there exists $R>0$ such that for all $x\in\Rd$ verifying $|x|>R$ we have 
\begin{equation}\label{eq:(3.6)DengLi}
\frac{\tau w(x)-f(x)}{w(x)}\geq \frac{\tau}{2}.
\end{equation}
Indeed, setting $T(x):=(|V'_1|+|V'_2|+|W'_1|+|W_2'|)$, by Claim~\ref{cl:2} we have 
\[
f \leq T(x)(|u'|+|v'|)\leq CT(x)(\tilde u+\tilde v)=CT(x)w.
\]
Therefore
\[
\frac{\tau w(x)-f(x)}{w(x)}\geq \tau-CT(x).
\]
By (VW2), we can take $R$ large enough so that  $CT(x)\leq \frac{\tau}{2}$ for $|x|>R$, which proves \eqref{eq:(3.6)DengLi}.

Note that $w\geq 0$ by definition. Since $w\in\mathcal{C}^0(\Rd)\cap W^{2,q}(\Rd)$, there exists $C_R$ such that for all $x\in\Rd$ with $|x|<R$ we have
\[
0\leq w(x)\leq C_R.
\]
Define $\psi(x):=C_Re^{-\sqrt{\frac{\tau}{2}}(|x|-R)}$. It is easy to see that
\begin{gather}
-\Delta \psi+\frac{\tau}{2}\psi\geq 0\mbox{ on }\Rd\setminus\{0\},\label{eq:Helmholtzforpsi}\\
w(x)-\psi(x)\leq 0\mbox{ on }\{x\in\Rd, |x|<R\}.\nonumber
\end{gather}
Therefore we only have to prove that $w(x)\leq \psi(x)$ for $|x|>R$. We proceed by contradiction. Assume that there exists $x_0\in\Rd$ with $|x_0|>R$ such that $w(x_0)>\psi(x_0)$. Define the set
\[
\Omega:=\{x\in\Rd, w(x)>\psi(x)\}.
\]
Then $\Omega$ is a non-empty open set, for all $x\in\Omega$ we have $|x|>R$ and for all $x\in\partial\Omega$ we have $w(x)-\psi(x)=0$. On $\Omega$, by \eqref{eq:Helmholtzforw}, \eqref{eq:(3.6)DengLi} and \eqref{eq:Helmholtzforpsi} we have
\begin{align*}
\Delta(w-\psi)&=\Delta w-\Delta\psi=\tau w-f-\Delta\psi\\
&=\frac{\tau w-f}{w}w-\Delta\psi\geq \frac{\tau}{2}(w-\psi)>0.
\end{align*}
By the maximum principle, this implies that $w-\psi\leq0$ on $\Omega$, a contradiction. Thus, for all $x\in\Rd$ we have
\[
w(x)\leq \psi(x)=C_Re^{-\sqrt{\frac{\tau}{2}}(|x|-R)}=C_Re^{\sqrt{\frac{\tau}{2}}R}e^{-\sqrt{\frac{\tau}{2}}|x|}=Ce^{-\sqrt{\frac{\tau}{2}}|x|}.
\]
This ends the proof.
\end{proof}

\begin{proof}[Proof of Proposition~\ref{thm:decayofeigenfunction}]
The statement is an immediate consequence of Lemma~\ref{lem:regularityeigenfunctions}, Claim~\ref{cl:2} and Lemma~\ref{lem:decay}.
\end{proof}

\subsection{Higher regularity and decay}

Upon assuming more regularity and decay, we can obtain more regularity and decay on the solutions to $(L-\lambda I)=A$.

 The new assumption is the following.
\begin{itemize}
\item[(VW3)] $V_1,V_2,W_1,W_2\in\q H(\C)$.
\end{itemize}
Recall that $\q H$ was defined in \eqref{eq:defspaceqH}.
\begin{prop}\label{prop:decayreg}
Assume that (VW1)-(VW3) hold.
\begin{itemize}
\item[(i)]
 Let $\lambda,u$ and $v$ be as in Proposition~\ref{thm:decayofeigenfunction}. Then $u,v\in \q H(\C)$. 
\item[(ii)]
Let $\lambda\notin \Sp(L)$ and take  $A\in \q H(\C^2)$. Then there exists $X\in \q H(\C^2)$ such that $(L-\lambda\Id)X=A $.
\end{itemize}
\end{prop}

\begin{proof}
(i) The assertion follows from similar arguments to those used in the proof of Proposition~\ref{thm:decayofeigenfunction}, provided we remark that (using the same notations)
$D^a u'=g^d_{-\w+\lambda'}*D^af_1$, $D^av'=-g^d_{-\w-\lambda'}*D^af_2$ and $D^af_1, D^af_2$ satisfy the same properties as $f_1$ and $f_2$.

(ii) Since $\lambda\notin\Sp(L)$ the operator $L-\lambda\Id$ is invertible, hence the existence of $X\in H^2(\Rd,\C^2)$ such that $(L-\lambda\Id)X=A $. Regularity of $X$ follows from a standard bootstrap argument as explained in the proof of Proposition~\ref{thm:decayofeigenfunction} (ii). We now recall that $L=-iP^{-1}L'P$.
Hence, if we define $X'=PX$, $\lambda'=i\lambda$, and $A'=iPA$ then 
\[
 (L'-\lambda'\Id)X'=A'.
\]
Recall that $L'-\lambda'I=H+K$. Set $Y=\begin{pmatrix}y_1\\y_2\end{pmatrix}:=KX'$ and $A'=\begin{pmatrix}a_1\\a_2\end{pmatrix}$.
Then we can represent $X'=\begin{pmatrix}x_1\\x_2\end{pmatrix}$ in the following way
\[
x_1=g^d_{-\w+\lambda'}*(y_1+a_1)\mbox{ and }x_2=-g^d_{-\w-\lambda'}*(y_2+a_2).
\]
The terms $g^d_{-\w+\lambda'}*a_1$ and $g^d_{-\w-\lambda'}*a_2$ are clearly exponentially decaying, with decay rate $\alpha$. Since $V_1,V_2,W_1,W_2\in\q H(\C^2)$, it follows that each component of $Y$ is also exponentially decaying with rate $\alpha$. Hence $g^d_{-\w+\lambda'}*y_1$ and $g^d_{-\w-\lambda'}*y_2$ are exponentially decaying with decay rate $\alpha$. The decay rate of the derivatives of $X'$ is follows immediately if we remark that for any multiindex $a$ we have
$D^a x_k=g^d_{-\w+\lambda'}*D^a(y_k+a_k)$ for $k=1,2$. 
\end{proof}

\section{Instability of solitons and multi-solitons}\label{ap:cor}

Since \eqref{eq:nls} is Galilean invariant, we can assume in this section without loss of generality that $v_1=x_1=\gamma_1=0$. Hence $R_1(t,x)=e^{i\w_1 t}\Phi_1(x)$.

Recall that, as defined in Section \ref{sec:profile}, $Y(t)$ is of the form $e^{-\rho t} (\cos(\theta t) Y_1(x) + \sin (\theta t)
  Y_2(x))$, where $Y_1, Y_2$ are smooth, exponentially decaying functions,
  along with their derivatives. 
Notice that if $u(t,x)$ is a solution to \eqref{eq:nls} and $T,\vartheta \in \m R$, then so is $\bar u (T-t,x)e^{i \vartheta}$. The hypotheses of Theorem \ref{thm:2} are verified by $\Phi_1$ and therefore also by $\bar\Phi_1$. Hence the conclusion of Theorem \ref{thm:2} holds for $\tilde R_1(t,x):=\bar{R}_1(-t,x)=e^{i\w_1 t}\bar\Phi_1$.
Let ${\mathfrak{u}} \in \q C([T_0, \infty), H^1(\m R^d)$ be the solution constructed
  in Theorem \ref{thm:2} associated with the soliton $\tilde R_1(t,x)$
and correction $e^{-\rho t} (\cos ( \theta t) Y_1(x) + \sin (\theta t) Y_2(x)) + O(e^{-2 \rho t})$ (i.e. ${\mathfrak{u}}=u_1$ in the notations of Theorem \ref{thm:2}). In particular, for all $\sigma \ge 0$,
\[
 \forall t \ge T_0, \quad \| {\mathfrak{u}}(t) - \tilde R_1(t) - Y(t) \|_{H^\sigma(\m R^d)} \le C e^{-2\rho t}. 
\]
Note that we construct ${\mathfrak{u}}$ on $\tilde R_1$ and not $R_1$ so as to have instability forward in time.

\subsection{Orbital instability of one soliton}

First let us prove a modulation lemma.
\begin{lem} \label{varpi:loc}
There exist $\e>0$, $t_0 \ge T_0$ and $M \ge 0$ such that
\[ \displaystyle \inf_{y\in\Rd, \vartheta\in\R} \| {\mathfrak{u}} (t_0) - \bar\Phi_1(x-y) e^{i \vartheta} \|_{L^2(B(0,M))} = \e >0. \] 
\end{lem}

\begin{proof} Let $t_0>T_0$ to be determined later. Up to increasing $t_0$, we can assume that $\w_1t_0\equiv 0(2\pi)$.

Consider $\Theta(y,\vartheta) = \| {\mathfrak{u}}(t_0) - \bar\Phi_1(x-y) e^{i \vartheta} \|_{L^2(\m R^d)}$. The function $\Theta$ is continuous on $\m R^{d+1}$. Notice that for $\vartheta = 0$ and $y = 0$, one gets $\Theta (0,0) \le C e^{-\rho t_0}$.

Now, we have that $ \liminf_{|y| \to \infty} \inf_{\vartheta\in\R} \Theta (y,\vartheta) \ge 2
\| \bar\Phi_1 \|_{L^2(\m R^d)} - C e^{-\rho t_0}$ due to space localization of
$\bar\Phi_1$, so that, as $\vartheta \in \m R/2\pi\m Z$ compact, if $t_0$ is large
enough, $\inf_{y\in\Rd,\vartheta\in\R} \Theta (y,\vartheta)$ is attained at some point $(y_0, \vartheta_0)$.

Assume $\Theta (y_0,\vartheta_0)=0$, i.e. ${\mathfrak{u}}(t_0) = \bar\Phi_1(x-y_0) e^{i\vartheta_0}$. 

\noindent\emph{Claim:} There exists a continuous function $\eta$ such that $\eta(0)=0$ and $|y_0| + |\vartheta_0| \le \eta( e^{-\rho t_0})$. 

Indeed, first consider $y_0$. Denote $g(y) = \| |\bar\Phi_1| - |\bar\Phi_1(\cdot -y) | \|_{L^2(\m R^d)}^2$. We have
\[ 0 = \Theta (y_0, \vartheta_0) \ge \| |{\mathfrak{u}}(t_0)| - |\bar\Phi_1( \cdot -y_0)| \|_{L^2(\m R^d)} \ge \| |\bar\Phi_1| - |\bar\Phi_1(\cdot-y_0)| \|_{L^2(\m R^d)} - C \| Y(t_0) \|_{L^2(\m R^d)}. \]
As $\| Y(t_0) \|_{L^2(\m R^d)} \le C e^{-\rho t_0}$, we get $g(y_0) \le C^2 e^{-2 \rho t_0}$. Now, due to space localization of $\bar\Phi_1$, $g(y) \to 2 \| \bar\Phi_1 \|_{L^2(\m R^d)}^2 > 0$ as $|y| \to+ \infty$. Let $(y_n)$ be such that $g(y_n) \to 0$, and $y_n \not \to 0$. Then up to a subsequence, $y_n \to y^\infty$ and $g(y^\infty)=0$, so that $|\bar\Phi_1|$ is periodic and as $\bar\Phi_1 \in L^2(\m R^d)$, $\bar\Phi_1 \equiv 0$, a contradiction. This shows that $y \to 0$ as $g(y) \to 0$, and it gives the bound on $y_0$. For $\vartheta_0$,
\[ 0 = \| {\mathfrak{u}}(t_0) - \bar\Phi_1( \cdot -y_0) \|_{L^2(\m R^d)} \ge -\| {\mathfrak{u}} (t_0) - \bar\Phi_1 \|_{L^2(\m R^d)} + \| \bar\Phi_1 - \bar\Phi_1 e^{i \vartheta_0} \|_{L^2(\m R^d)} - \| \bar\Phi_1 - \bar\Phi_1( \cdot -y_0) \|_{L^2(\m R^d)}, \]
As $\| \bar\Phi_1 - \bar\Phi_1 e^{i \vartheta_0} \|_{L^2(\m R^d)} = |1 - e^{i\vartheta_0}|\nld{\bar\Phi_1}$, we deduce that $|\vartheta_0| \le C e^{-\rho t_0} + C g(y_0)$. This concludes the proof of the claim.

Denote $T_{\bar\Phi_1} \q F$
the tangent space of $\q F = \{ \bar\Phi_1(\cdot - y) e^{i\vartheta} | (y,\vartheta)
\in \m R^d \}$ at point $\bar\Phi_1$. Note that, due to the \emph{Claim}, $\q F$ is a manifold. It is easy to see that $T_{\bar\Phi_1} \q F \subset \ker
L_{\m C}$ (by differentiating the relation $\Delta \bar\Phi_1(x-y) +
g(|\bar\Phi_1(x-y)|^2)\bar\Phi_1(x-y) = \omega_1 \bar\Phi_1(x-y)$).
But for all $t$, $(\cos(\theta t) Y_1(x) + \sin (\theta t)
Y_2(x)) \notin \ker L_{\m C}$ (as $Y_1,Y_2$ are build on an eigenvector for an eigenvalue of positive real part of $L_\C$). As ${\mathfrak{u}}(t_0)
= \bar\Phi_1 + e^{\rho t_0} (\cos(\theta t_0) Y_1(x) + \sin (\theta t_0) Y_2(x)) +
O(e^{-2\rho t_0})$, up to choosing $t_0+2k \pi/\theta$, ($k \in \m N$ large)
instead of $t_0$, this proves that ${\mathfrak{u}}(t_0) \notin \q F$. We proved
that for $t_0$ large enough,
\[ 
\inf_{y\in\Rd,\vartheta\in\R} \| {\mathfrak{u}}(t_0) - \bar\Phi_1(x-y) e^{i\vartheta} \|_{H^1(\m R^d)} > 0. 
\]
Assume that this does not hold when we restrict to $L^2(B(0,M))$, for any large $M$. This would mean that for all $m \ge 0$, there exist $y_m\in\Rd, \vartheta_m\in\R$ such that 
\[ \| {\mathfrak{u}}(t_0) - \bar\Phi_1(x-y_m) e^{i\vartheta_m} \|_{L^2(B(0,m))} \le \frac{1}{m}. \]
Then by localization arguments, $(y_m)$ remains bounded, so that up to a subsequence, $y_m \to y_\infty$, $\vartheta_m \to \vartheta_\infty$. Therefore  $\| {\mathfrak{u}}(t_0,x) - \bar\Phi_1(x-y_\infty) e^{i\vartheta_\infty} \|_{L^2(\Rd)} =0$, so that ${\mathfrak{u}}(t_0,x) = \bar\Phi_1(x-y_\infty) e^{i\vartheta_\infty}$, a contradiction.
\end{proof}

\begin{proof}[Proof of Corollary \ref{cor1}]
Let $t_0$ and $\e$ be given by Lemma \ref{varpi:loc}. Take an increasing sequence $(S_n)$ so that $S_n \to +\infty$ as $n\to+\infty$, and define $T_n := S_n - t_0$ and 
\[ 
u_n(t,x) := \bar{\mathfrak{u}} (S_n-t , x) e^{-i\omega_1 S_n}. 
\]
Then $u_n \in \q C([0,T_n], H^1(\m R^d))$ is a solution of \eqref{eq:nls}, and 
\begin{gather*}
\begin{aligned}
u_n(0,x) & =  \bar{\mathfrak{u}} (S_n, x ) e^{-i \omega_1 S_n}= \Phi_1(x ) +O_{H^\sigma} (e^{-\rho S_n}),\\
u_n (T_n,x) &=\bar {\mathfrak{u}}(t_0,x) e^{-i\omega_1 S_n}.
\end{aligned}
\end{gather*}
Therefore, 
$
\| u_n (0) - R_1(0) \|_{H^\sigma(\Rd)} \le C e^{-\rho S_n} \to 0
$
as $n\to+\infty$.
Due to Lemma \ref{varpi:loc}, we deduce that for all $n\in\N$ we have
\[ \inf_{y\in\Rd, \vartheta\in\R} \| u_n(T_n) - e^{i \vartheta } \Phi_1(\cdot-y) \|_{\ld}  \ge \inf_{y\in\Rd, \vartheta\in\R} \| {\mathfrak{u}}(t_0) - \bar\Phi_1(x-y) e^{i \vartheta } \|_{L^2(B(0,M))}  \ge \e, \]
which is the desired conclusion.
\end{proof}

\subsection{Instability of multi-solitons}

\begin{proof}[Proof of Corollary \ref{cor2}]
Let $T>0$, $M$ be given by Lemma \ref{varpi:loc} and $\e$, $(u_n)$, $(T_n)$ be given by Corollary \ref{cor1}. 

The idea is the following. We use the fact that $u_n(T_n)$ is $\e$-away from the orbit of the soliton $R_1$.
Given a parameter $I$, we consider at time $I$ an initial data $w(I)$ which is $u_n(0)$ adequately shifted, denoted by $\tilde u_n(I)$, plus the sum of the $R_j(I)$, $j \ge 2$. (All the functions will depend on $n$ and $I$, although we do not always show this dependence for convenience in the notation). We aim at controlling $w$ up to time $I+T_n$. The role of $I$ is to ensure that the interaction of $u_n$ and the $R_j$ are small: as $\{ u_n(t) | t \in [0,T_n] \}$ is compact and the $R_j(t)$ ($j\ge 2$) are localized away from $\tilde u_n(t)$, their interaction goes to $0$ as $I \to +\infty$. Using a Gronwall type argument, we are able to show that $w(I+T_n)$ is $\tilde u_n(I+T_n) + \sum_{j=2}^N R_j(I+T_n) + o_{I \to +\infty}(1)$. As $u_n(T_n)$ is $\e$-away from the soliton family, we deduce that $w(I+T_n)$ is $\e - o_{I\to +\infty}(1) \ge \e/2$ away from the family of a sum of solitons.

Given $I \ge T$, define $\tilde u \in \q C([I,I+T_n], H^1(\m R^d))$ by
\[ 
\tilde u_{n}(t,x) = u_n(t-I,x).
\]
Possibly increasing $I$ so that $\w_1I=0(2\pi)$, we have $\| \tilde u_{n}(I) - R_1(I) \|_{H^\sigma(\Rd)} = \| u_n(0) - R_1(0) \|_{H^\sigma(\Rd)} \to 0$ as $n \to +\infty$ and $\tilde u_{n}(I+T_n)$ is $\e$-away from the $\Phi_1$-soliton family. Consider the solution $w_n \in  \q C([I,T^*), H^1(\m R^d))$ to \eqref{eq:nls} with initial data at time $I$
\[ 
w_n(I,x) = \tilde u_{n}(I,x) + \sum_{j =2}^N R_j(I,x).
\]
If $T^*<+\infty$, the blow-up alternative for \eqref{eq:nls} automatically implies instability on the multi-soliton, hence we assume $T^*=+\infty$. 
Let $\sigma > d/2$ be an integer. Notice that, as $u_n \in \q C([0,T_n],H^\sigma(\Rd))$ and $[0,T_n]$ is compact, the set $\left\{ u_{n}(t) \middle| t \in [0,T_n] \right\}$ is compact in $H^\sigma(\Rd)$.
In particular, $\sup_{t \in [0,T_n]} \| u_n(t) \|_{H^\sigma(|x| \ge R)} \to 0$ as $R \to +\infty$. Hence, as the $R_j$ are decoupling as time grows, there exists a function $\eta(I)$ such that $\eta(I) \to 0$ as $I \to +\infty$ and
\[
\forall t \in [I,I+T_n], \quad \sum_{j \ge 2} \| \tilde u_n(t) R_j(t) \|_{H^\sigma} \le \eta(I).
\]
Denote $x_j(t) = v_j t+ x_j$. Up to modifying the function $\eta$, we can also assume that the $R_j(t)$, $j \ge 2$, are far away from $x_1(t)\equiv0$, and that the multisoliton $R(t)$ is near the sum of solitons $\sum_{j=1}^N R_j(t)$, that is
\[
\forall t \ge I, \quad \sum_{j=2}^N \| R_j(t) \|_{H^\sigma(B(0, 2M))} + \| R(t) - \sum_{j=1} R_j(t) \|_{H^\sigma(\Rd)} \le \eta(I).
\]
Finally we denote ${J} = {I}+T_n $ and 
\[
z(t) = w_n(t) - (\tilde u_{n}(t) + \sum_{j =2}^N R_j(t)). 
\]
Now, as $f$ is $\q C^\infty$, for all $R>0$, there exists $C(R)$ such that
\begin{equation} \label{fbound}
\forall a,b \in B(0,R), \quad |f(a+b) - f(a) - f(b) | \le C(R) |a| |b|.
\end{equation}
Indeed, this expression is symmetric in $a,b$, so that we can assume without loss of generality that $|b| \le |a|$. As $f(0) = f'(0) =0$, we have that $|f(b)| \le C |b|^2 \le C|a| |b|$, and a Taylor expansion shows that 
\[
| f(a+b) - f(a)| = b \int_0^1 |f'(a+tb)| dt \le b \sup_{x \in B(0,|a|+|b|)} |f'(x)| \le C |b| (|b| + |a|) \le C|a||b|. \]
Now, as $H^\sigma(\Rd)$ is an algebra, we deduce from \eqref{fbound} that there exists a constant $C>0$ (depending only on the $\Phi_j$) such that for $t \in [{I},{J}]$,
\[ 
\| f(z(t) + \tilde u_{n}(t) + \sum_{j =2}^N R_j(t)) - f( \tilde u_{n} ) - \sum_{j =2}^N f(R_j(t)) \|_{H^\sigma(\Rd)} \le C \| z(t) \|_{H^\sigma(\Rd)} + C \sum_{j=2}^N \| \tilde u_{n}(t) R_j(t) \|_{H^\sigma(\Rd)}. 
\]
The function $z$ satisfies the equation 
\[
i z_t + \Delta z + f \bigg( z + \tilde u_{n} + \sum_{j=2}^N R_j \bigg) - f(\tilde u_{n}) - \sum_{j=2}^N f(R_j) =0,
\]
Since $z({I})=0$, Duhamel formula for $z$ gives
\[
z(t) = \int_{{I}}^t e^{i\Delta(t-s)} \Bigg( f \bigg( z(s) + \tilde u_{n}(s) + \sum_{j=2}^N R_j(s) \bigg) - f(\tilde u_{n}(s)) - \sum_{j=2}^N f(R_j(s)) \Bigg) ds.
\]
Hence, for all $t \in [{I},{J}]$
\[ 
\| z(t) \|_{H^s(\Rd)} \le C \int_{{I}}^t (\| z(s) \|_{H^\sigma(\Rd)} + \eta) ds \le C \int_{{I}}^t \| z(s) \|_{H^\sigma(\Rd)} ds + \eta(I) (t-{I}). 
\]
By Gr\"onwall's Lemma, we deduce that for $t \in [{I}, {J}]$, we have
\[ 
\| z(t) \|_{H^\sigma(\Rd)} \le C \eta(I) (t-{I})e^{C(t-{I})}\le C_n \eta(I), 
\]
where $C_n = CT_n e^{CT_n}$. Thus for all $n \in \m N$ we have
\[
\left\| w_n({J}) - u_n(T_n)  - \sum_{j=2}^N R_j({J}) \right\|_{H^\sigma(\Rd)} \le C_n \eta(I).
\]
Now choose $I_n$ such that $C_n \eta(I_n) \le \e/3$ and set $J_n=I_n+T_n$. Then $\|z(J_n)\|_{H^\sigma(\Rd)} \le \e/3$. Then, given $y_j \in \m R^d$, $\vartheta_j \in \m R$, we have (denote $c_j=c_j(t)=-\frac{1}{4}|v_j|^2t+\w_jt+\g_0$)
\begin{align*}
\MoveEqLeft \| w_{n}(J_n) - \sum_{j=1}^N \Phi_j(\cdot - y_j) e^{i(\frac{1}{2}v_j\cdot
  x +\vartheta_j)} \|_{L^2} \\
& \ge \left\| u_n(T_n) + \sum_{j=2}^N R_j(J_n) - \sum_{j=1}^N \Phi_j(\cdot - y_j) e^{i (\frac{1}{2}v_j\cdot
  x+\vartheta_j)} \right\|_{L^2}
- \left\| w_{n}(J_n) - u_n(T_n) - \sum_{j=2}^N R_j(J_n) \right\|_{L^2}  \\
& \ge \left\| u_n(T_n) - \Phi_1 (x-y_1) e^{i \vartheta_1} + \sum_{j=2}^N (\Phi_j(x-x_j(J_n)) e^{i(\frac{1}{2}v_j\cdot
  x+c_j)} - \Phi_j(x- y_j) e^{i (\frac{1}{2}v_j\cdot
  x+\vartheta_j)}) \right\|_{L^2} - \e/3.
\end{align*}
Now consider $y_j$, $\vartheta_j$ that realize a near infimum, say $\| w_{n}(J_n) - \sum_{j=1}^N \Phi_j(\cdot - y_j) e^{i (\frac{1}{2}v_j\cdot
  x+\vartheta_j)} \|_{L^2} \le 2\e$.
Then considering the $L^2$ norm on balls $B(x_j(J_n),R)$ around each exited state $R_j$, $j \ge 2$ (for some large and fixed radius $R$), we see that,
 up to a permutation if two $\Phi_j$ or more are equal, we must have $y_j - x_j(J_n) = O(1)$ for $j \ge 2$. In particular, this implies that
\[ \left\| \sum_{j=2}^N (\Phi_j(x-x_j(J_n)) e^{i(\frac{1}{2}v_j\cdot
  x+c_j)} - \Phi_j(x- y_j) e^{i (\frac{1}{2}v_j\cdot
  x+\vartheta_j)} \right\|_{H^\sigma(B(0,M))} = o_{I_n \to +\infty}(1) \le \e/3, \]
up to increasing again $I_n$. Thus,
\begin{align*}
\inf_{\substack{ y_j\in\Rd, \vartheta_j\in\R,\\j=1,...,N}} \| w_{n}(J_n) - \sum_{j=1}^N \Phi_j(\cdot - y_j) e^{i (\frac{1}{2}v_j\cdot
  x+\vartheta_j)} \|_{L^2}
& \ge \| w_{n}(J_n) - \sum_{j=1}^N \Phi_j(\cdot - y_j) e^{i (\frac{1}{2}v_j\cdot
  x+\vartheta_j)} \|_{L^2(B(0,M))} \\
& \ge  \| u_n(T_n)  - \Phi_1 (x-y_1) e^{i \vartheta_1} \|_{L^2(B(0,M))} - 2\e/3 \\
& \ge \e - 2\e/3 \ge \e/3,
\end{align*}
where we used Corollary \ref{cor1} on the last line. As 
\[ 
\| w_n(I_n) - R(I_n) \|_{H^\sigma(\Rd)} \le \| w_n(I_n) - \sum_{j=1}^N R_j(I_n) \|_{H^\sigma(\Rd)} + \| \sum_{j=1}^N R_j (I_n) - R(I_n) \|_{H^\sigma(\Rd)} \to 0, 
\]
$w_n$, $I_n$ and $J_n$ satisfies the conditions of Corollary \ref{cor2}.
\end{proof}

\begin{rmk}
Notice that we did not use any high speed condition on the $v_j$. 
The most delicate point here is that we have no uniform spatial decay on $u_n$ (as well as on the multi-soliton constructed in Theorem \ref{thm:3}), apart that coming from $H^\sigma(\Rd)$ compactness. 
We conjecture it should be exponentially decaying (in space) around the soliton (resp. every soliton $R_j$); a proof of this should be related to uniqueness of the multi-soliton in the $L^2$ sub-critical case, which is currently an open problem.
\end{rmk}

\section{Coercivity for a soliton}\label{ap:coercivity}
This section is devoted to the proof of Lemma~\ref{lem:coercivity-H-0}.

\begin{proof}[Proof of Lemma~\ref{lem:coercivity-H-0}]
We first remark that
$ R_0$ is solution of 
\begin{equation}\label{eq:eqsolevedbyR}
-\Delta  R_0 +\left(\w_0+\frac{|v_0|^2}{4}\right)R_0-f(R_0)+iv_0\nabla R_0 =0.
\end{equation}
Therefore it is a critical point of the functional $ \tilde{S}_0$ defined for $w\in\hu$ by
\[
\tilde{S}_0(w):=\frac{1}{2}\nldd{\nabla w}+\frac{1}{2}\left(\w_0+\frac{|v_0|^2}{4}\right)\nldd{w}
-\intrd F(w)dx-\frac{1}{2}v_0\cdot\Im\intrd\bar{w}\nabla wdx.
\]
The quadratic form $H_0$ is precisely
\[
H_0(t,w)= \dual{\tilde{S}''_0(R_0)w}{w}. 
\]
Consider $z$ such that $w=e^{-i(\frac{1}{2}v_0\cdot x-\frac{1}{4}|v_0|^2t+\w_0t+\g_0)}z(x+v_0t+x_0)$. 
Then it is easy to see that
\[
H_0(t,w)=\tilde{H}_0(z):=
\nldd{\nabla z}+\w_0\nldd{z}
-\intrd
\left(
g(|\Phi_{0}|^2)|z|^2
+2g'(|\Phi_{0}|^2)\Re(\Phi_{0}\bar{z})^2\right)
dx.
\]
It is well-known that up to a finite number of non-positive directions $\tilde{H}_0(z)$ controls the $\hu$-norm of $z$. Indeed, the self-adjoint operator corresponding to the quadratic form $\tilde{H}_0$ (viewed on $H^1(\Rd,\R^2)$) is a compact perturbation of $\begin{pmatrix}-\D+\w_0&0\\0&-\D+\w_0\end{pmatrix}$, hence its spectrum lies on the real line and its essential spectrum is $[\w_0,+\infty)$. Since in addition the quadratic form $\tilde{H}_0$ is bounded from below on the unit $\ld$-sphere, the corresponding operator admits only a finite number of eigenvalues in $(-\infty,\w_0')$ for any $\w_0'<\w_0$. In particular, there exist $\tilde K_0>0$, $\nu_0\in\N$ and $\tilde X^1_0,...,\tilde X^{\nu_0}_0\in\ld$ such that $\nld{\tilde X_0^k}=1$ for any $k$ and
\[
\nhud{z}\leq \tilde K_0\tilde{H}_0(z)+\tilde K_0\sum_{k=1}^{\nu_0}\psld{z}{\tilde X^k_0}^2.
\]
Since 
\[
\nldd{\nabla w}=\frac{3}{2}\nldd{\nabla z}+\frac{3|v_0|^2}{4}\nldd{z}
\]
there exists $ K_0>0$ such that 
\[
\nhud{w}\leq  K_0 H_0(t,w)+ K_0\sum_{k=1}^{\nu_0}\psld{w}{ X^k_0(t)}^2,
\]
where $ X^k_0(t):=e^{i(\frac{1}{2}v_0\cdot x-\frac{1}{4}|v_0|^2t+\w_0t+\g_0)}\tilde X^k_0(x-v_0t-x_0)$.
\end{proof}

%


\def\cprime{$'$}

\end{document}